\documentclass[tbtags,reqno]{amsart}

\usepackage{geometry}
\usepackage[normalem]{ulem}%Agregado para tachar, para comentarios

\usepackage{graphicx, subfigure}
\usepackage{epstopdf}
\usepackage{caption}
\usepackage{pythonhighlight} %Para codigo

\usepackage{xfrac}
\usepackage{tikz-cd}
\usetikzlibrary{decorations.pathmorphing}

\usepackage{mathtools}
\usepackage[...]{youngtab}
\usepackage[boxsize=1.25em, centerboxes]{ytableau}
\newsavebox\yTAAB
\newsavebox\yTAAC
\newsavebox\yTABB
\newsavebox\yTACB
\newsavebox\yTABC
\newsavebox\yTACC
\newsavebox\yTBBC
\newsavebox\yTBCC
\newsavebox\yTADB
\newsavebox\yTADC
\newsavebox\yTADD
\newsavebox\yTBDC
\newsavebox\yTBDD
\newsavebox\yTAAD
\newsavebox\yTABD
\newsavebox\yTBBD
\newsavebox\yTACD
\newsavebox\yTBCD
\newsavebox\yTCCD
\newsavebox\yTCDD

\newsavebox\yTTABC
\newsavebox\yTTACB
\newsavebox\yTTBAC
\newsavebox\yTTBCA
\newsavebox\yTTCAB
\newsavebox\yTTCBA

\newsavebox\yTTABCw
\newsavebox\yTTACBw
\newsavebox\yTTBACw
\newsavebox\yTTBCAw
\newsavebox\yTTCABw
\newsavebox\yTTCBAw

\usepackage{cleveref} %enlaces

\usepackage{mathrsfs}
\usepackage{latexsym}
\usepackage{amsthm,amsopn,tabmac,amsfonts,amssymb,epsfig,color,pstricks,pb-diagram,mathdots,stmaryrd}
\usepackage[active]{srcltx}
\numberwithin{equation}{section}
\makeatletter
\renewcommand{\subsubsection}{\@startsection
{subsubsection}
{3}
{0mm}
{\baselineskip}
{-0.5\baselineskip}
{\normalfont\normalsize\bfseries}}
\makeatother

\newtheorem{theorem}{Theorem}
\newtheorem{lemma}[theorem]{Lemma}
\newtheorem{proposition}[theorem]{Proposition}
\newtheorem{example}[theorem]{Example}

\newtheorem{corollary}[theorem]{Corollary}

\newtheorem{definition}[theorem]{Definition}

\newtheorem{remark}[theorem]{Remark}

\newcommand{\cercle}[1]{\ensuremath{\setlength{\unitlength}{1ex}\begin{picture}(2.8,2.8)\put(1.4,1.4){\circle{2.8}\makebox(-5.6,0){#1}}\end{picture}}}

\newcommand{\tcercle}[1]{\ensuremath{\setlength{\unitlength}{1ex}\begin{picture}(2.8,2.8)\put(1.4,1.4){\circle{2.8}\makebox(-5.6,0){#1}}\end{picture}}}

\title{A new characterization of right keys, and the $m$-symmetric Schur functions at $t=0$}
\begin{document}

\author{Luc Lapointe}
\address{Instituto de Matem\'aticas, Universidad de
Talca, 2 norte 685, Talca, Chile.}
\email{llapointe@utalca.cl }

\author{Luis Pena}
\address{Instituto de Matem\'aticas, Universidad de
Talca, 2 norte 685, Talca, Chile.}
\email{luis.cardenas@utalca.cl }

\begin{abstract}
The ring $R_m$ of $m$-symmetric functions consists of the formal power series that are symmetric in the variables $x_{m+1},x_{m+2},\dots$ but carry no symmetry in the first $m$ variables. We develop a combinatorial theory for the specialization at $t=0$ of the Schur functions of $R_m$.
Our main tool is a new characterization of right key tableaux as suprema of the sets of decreasing subwords of the reading words of the subtableaux of $T$.  Being invariant under elementary Knuth transformations, this characterization is compatible with the RSK correspondence.  We obtain in this way a generating function over semistandard tableaux for the $m$-symmetric Schur functions at $t=0$, together with a combinatorial proof of a Cauchy identity in $R_m$. The $m$-symmetric Schur functions and their dual are then respectively identified with Demazure atoms and Demazure characters.  Restricted to the last $m$ variables, our correspondence specializes to a proof, by ordinary RSK, of Lascoux's nonsymmetric Cauchy identity for Demazure characters and atoms.  As further applications, we relate the $m$-symmetric Schur functions at $t=0$ to the almost symmetric Schur functions through a unitriangular change-of-basis matrix, obtain tableau generating functions and Cauchy identities for both families, and derive Jacobi-Trudi type determinantal formulas for three different bases.

\end{abstract}
  

\thanks{Funding: this work was supported by
the Fondo Nacional de Desarrollo Cient\'{\i}fico y Tecnol\'ogico de Chile (FONDECYT) Regular Grant \#1250376 and by the Beca de Doctorado Nacional ANID \#21211108.}

\maketitle

\savebox\yTAAB{\tableau[scY]{ 1 & 1 \\ 2
 } }%
\savebox\yTABB{\tableau[scY]{ 1 & 2 \\ 2
 } }%
\savebox\yTAAC{\tableau[scY]{ 1 & 1 \\ 3
 } }%
\savebox\yTACB{\tableau[scY]{ 1 & 3 \\ 2
 } }%
\savebox\yTABC{\tableau[scY]{ 1 & 2 \\ 3
 } }%
\savebox\yTACC{\tableau[scY]{ 1 & 3 \\ 3
 } }% 
\savebox\yTBBC{\tableau[scY]{ 2 & 2 \\ 3
 } }% 
\savebox\yTBCC{\tableau[scY]{ 2 & 3 \\ 3
 } }%  
\savebox\yTADB{\tableau[scY]{ 1 & 4 \\ 2
 } }%  
\savebox\yTADC{\tableau[scY]{ 1 & 4 \\ 3
 } }%  
\savebox\yTADD{\tableau[scY]{ 1 & 4 \\ 4
 } }%   
\savebox\yTBDC{\tableau[scY]{ 2 & 4 \\ 3
 } }%    
\savebox\yTBDD{\tableau[scY]{ 2 & 4 \\ 4
 } }%    
\savebox\yTCDD{\tableau[scY]{ 3 & 4 \\ 4
 } }%  
\savebox\yTAAD{\tableau[scY]{ 1 & 1 \\ 4
 } }%   
\savebox\yTABD{\tableau[scY]{ 1 & 2 \\ 4
 } }%   
\savebox\yTBBD{\tableau[scY]{ 2 & 2 \\ 4
 } }%  
\savebox\yTACD{\tableau[scY]{ 1 & 3 \\ 4
 } }%  
\savebox\yTBCD{\tableau[scY]{ 2 & 3 \\ 4
 } }%   
\savebox\yTCCD{\tableau[scY]{ 3 & 3 \\ 4
 } }%    

\savebox\yTTABC{\tableau[scY]{
1 & 3 & 4 \\
2 & 6 \\
5
}
}%    

\savebox\yTTABCw{\tableau[scY]{
3 & 1 & 2  \\
\bl &  4 & 5 \\
\bl & \bl & 6
}
}%    

\savebox\yTTACB{ \tableau[scY]{
\bl & \bl & 4 \\
1 & 3 & 6 \\
2 \\
5
}
}%    

\savebox\yTTACBw{ \tableau[scY]{
\bl & \bl & 2 \\
\bl & \bl & 5 \\
1   & 4   & 6 \\
3   & \bl  \\
\bl & \bl 
}
}%    

\savebox\yTTBAC{ \tableau[scY]{
\bl & 1 & 4 \\
2 & 3 \\
5 & 6
}
}%    

\savebox\yTTBACw{ \tableau[scY]{
\bl   & 1 & 2 \\
\bl & 4 & 5 \\
3 & 6 
}
}%    

\savebox\yTTBCA{ 
\tableau[scY]{
\bl & 1 & 4 \\
\bl & 3 & 6 \\
2 & 5
}
}%    

\savebox\yTTBCAw{ 
\tableau[scY]{
\bl & 2 & 5 \\
1   & 4 \\
3   & 6
}
}%    

\savebox\yTTCAB{ 
\tableau[scY]{
\bl & \bl & 1 \\
\bl & \bl & 4 \\
2 & 3 & 6 \\
5 
}
}%    

\savebox\yTTCABw{ 
\tableau[scY]{
\bl & \bl & 2 \\
1   & 4   & 5 \\
3 \\
6 
}
}%    

\savebox\yTTCBA{ \tableau[scY]{
\bl & \bl & 1 \\
\bl & 3 & 4 \\
2 & 5 & 6 
}
}%    

\savebox\yTTCBAw{ \tableau[scY]{
1 & 2 & 5 \\
3 & 4 \\
6 
}
}%    

\section{Introduction} \label{secintro}

The Cauchy identity
\begin{equation} \label{usualCauchy}
    \frac{1}{ \prod_{i,j} (1- x_i y_j)  }   =\sum_{\lambda} s_\lambda(x) \, s_\lambda(y),
\end{equation}
where the sum is over all partitions, is one of the basic identities in the theory of symmetric functions.  It has a well-known combinatorial proof.  On the one hand, the Schur function $s_\lambda(x)$ is the generating function
$$
s_\lambda(x) = \sum_T x^T
$$
of the semistandard tableaux $T$ of shape $\lambda$.  On the other hand, the RSK correspondence is a bijection between biwords and pairs $(P,Q)$ of semistandard tableaux of the same shape.  Reading the weight of a biword in two different ways, on the left through the biword itself and on the right through the pair $(P,Q)$, gives \eqref{usualCauchy}.

There is a nonsymmetric counterpart to this picture.  The role of the Schur functions is played there by two families of polynomials which connect to representation theory: the key polynomials $K_\alpha(x)$, also known as Demazure characters \cite{Demazure1974a,Demazure1974b}, and the dual key polynomials $\hat K_\alpha(x)$, also known as Demazure atoms, which refine the former.  Both are indexed by weak compositions, and both are obtained from a dominant monomial by acting with an isobaric divided difference operator ($\pi_i$ or $\hat \pi_i$).  Lascoux and Schützenberger \cite{Lascoux1990a} gave a tableau model for the key polynomials in terms of the right key $K_+(T)$ of a tableau $T$, and this model has an obvious counterpart for the atoms.  Lascoux \cite{Lascoux2003} then established the nonsymmetric Cauchy identity
\begin{equation} \label{nonsymCauchy}
\frac{1}{\prod_{i+j \leq n+1}(1- x_i y_j)} = \sum_{\alpha} \hat K_{w_0 \alpha}(x) \, K_\alpha(y),
\end{equation}
in which the rectangle  of the classical kernel has been replaced by a staircase.  His proof proceeds through double crystal graphs, that is, by decomposing the set of biwords supported on the staircase into crystal components.  Several other proofs and extensions have appeared since: Fu and Lascoux \cite{FuLascoux2009} treat the other classical groups, Azenhas and Emami \cite{AzenhasEmami2015} obtain \eqref{nonsymCauchy} and its truncated-staircase versions through a restriction of Mason's skyline analogue of RSK \cite{Mason2009}, and Choi and Kwon \cite{ChoiKwon2018} give a bijective proof using Lakshmibai-Seshadri paths.

The present article is concerned with a family of rings that interpolates between the two settings just described.  For a nonnegative integer $m$, the ring $R_m$ of $m$-symmetric functions consists of the formal power series in $x_1,x_2,x_3,\dots$ that are symmetric in the variables $x_{m+1},x_{m+2},\dots$ but carry no symmetry in the first $m$ variables \cite{mSym}.  Thus $R_0$ is the ring of symmetric functions, while letting $m$ grow recovers, in the limit, the polynomials in infinitely many variables.  Bases of $R_m$ are indexed by $m$-partitions $\Lambda=(\pmb a;\lambda)$, a weak composition with $m$ parts followed by a partition, and the $t$-symmetrization of the nonsymmetric Macdonald polynomials yields a basis $P_\Lambda(x;q,t)$ of $m$-symmetric Macdonald polynomials \cite{mSym,ConchaLapointe}.

The $m$-symmetric Schur functions $s_\Lambda(x;t)$ that we study here were introduced in \cite{mSym} for the sake of the following conjecture.  Let $J_\Lambda(x;q,t)$ be the integral form of $P_\Lambda(x;q,t)$, and extend to $R_m$ the plethysm $X \mapsto X/(1-t)$ relevant to Macdonald polynomials.  The $m$-symmetric Macdonald positivity conjecture then asserts that
\begin{equation} \label{Kostkaintro}
  J_\Lambda\left[\frac{X}{1-t};q,t \right]=
  \sum_{\Omega} K_{\Omega \Lambda}(q,t) \, s_\Omega(x;t),
\end{equation}
with $K_{\Omega \Lambda}(q,t) \in \mathbb N[q,t]$.  For $m=0$ this is Macdonald's positivity conjecture \cite{Macdonald1995}, established by Haiman \cite{Haiman2001}, and the usual $(q,t)$-Kostka coefficients occur among the $K_{\Omega \Lambda}(q,t)$'s as special cases, so that \eqref{Kostkaintro} does extend the classical statement. 

The $m$-symmetric Schur functions are, however, only defined implicitly: one first constructs a dual family $s_\Lambda^*(x;t)$ out of multi-Schur functions and the Hecke algebra generators, and only then defines $s_\Lambda(x;t)$ by duality with respect to a scalar product on $R_m$.  No tableau generating series is available in general, and this is the gap that the present work fills at $t=0$.  We see this as a first step towards a resolution of the case $t=0$ of \eqref{Kostkaintro}.

The other specialization of the parameter $t$ that lends itself to an explicit combinatorial treatment is $t=1$.  The $m$-symmetric Schur functions at $t=1$ were studied in \cite{Positivityt1}, where they are the main ingredient in a proof of the $m$-symmetric Macdonald positivity conjecture at $t=1$.  We should stress that the combinatorics governing the case $t=0$ treated here, namely that of key tableaux, is of a rather different nature.

The starting point of this article is the following generalization of the Cauchy identity 
\begin{equation*} 
\frac{1}{\left[\prod_{i+j \leq m+1}(1- x_i y_j) \right] \left[\prod_{i,j} (1- x_i y_j) \right] }   =\sum_{\Lambda} s_\Lambda(x;0) \,  {\mathfrak s}_\Lambda^*(y), 
\end{equation*}
where  $s_\Lambda(x;0)$ is the $m$-symmetric Schur function at $t=0$, and where
$\mathfrak s_\Lambda^*(y)$ is the limit $t \to 0$ of a slightly modified version of a dual  $m$-symmetric Schur function.  As we will see, this identity easily follows from a reproducing kernel for a natural scalar product in the ring $R_m$ of $m$-symmetric functions.  Observe that its kernel is the product of the two kernels occurring in \eqref{usualCauchy} and \eqref{nonsymCauchy}: a staircase in the first $m$ variables, and a full rectangle in all of them.  This suggests that a proof should combine the symmetric and the nonsymmetric mechanisms rather than merely juxtapose them.

The goal of this article is to prove our  generalization of the Cauchy identity along the same lines as the classical argument recalled above.  Fortunately, the functions  ${\mathfrak s}_\Lambda^*(x)$
turn out to be key polynomials, and as such, have an expansion as a sum over tableaux.  On the other hand, the definition of the $m$-symmetric Schur functions $s_\Lambda(x;t)$ is not very explicit, and does not yield such a tableau expansion for the functions $s_\Lambda(x;0)$. To circumvent this problem, we define a new family of $m$-symmetric Schur functions $\mathfrak s_\Lambda(x)$
as a sum over tableaux, which will be related to dual key polynomials, and then prove a new version of our generalization of the Cauchy identity:
\begin{equation} \label{newversion} 
\frac{1}{\left[\prod_{i+j \leq m+1}(1- x_i \hat y_j) \right] \left[\prod_{i,j} (1- x_i y_j) \right] }   =\sum_{\Lambda} \mathfrak s_\Lambda(x) \,  {\mathfrak s}_\Lambda^*(y;\hat y),
\end{equation}
where $\hat y =\hat y_1,\dots, \hat y_m$ is  another alphabet, and where
$ {\mathfrak s}_\Lambda^*(y;\hat y)$ reduces to $ {\mathfrak s}_\Lambda^*(y)$
when $\hat y_i =y_i$ for $i=1,\dots,m$.
By comparing the two generalizations at $\hat y_i= y_i$, it is then immediate that
$\mathfrak s_\Lambda(x)=s_\Lambda(x;0)$, which gives the characterization of $s_\Lambda(x;0)$ as a sum over tableaux that we were seeking.

In order to prove the identity \eqref{newversion}, we need to obtain a bijection between families of biwords that we call admissible and pairs of tableaux $(P,Q)$ of the same shape that satisfy a certain admissibility condition.  The bijection is quite easy to get as it is provided by the usual RSK insertion (leading to the RSK correspondence).  But proving that it is indeed the correct bijection turns out not to be straightforward at all.  The difficulty is that admissibility is a condition on right keys, and that the right key is defined by a procedure --- an action of the symmetric group on the columns of a tableau --- that interacts poorly with a single insertion step.

Given the connection with key polynomials mentioned earlier, it is not surprising that the key tableaux play a fundamental role in our proof. As their usual definition is not very amenable to proving our bijection, we will have to introduce a new characterization better suited for our purposes.  It is only once equipped with this new definition that we will be able to prove  \eqref{newversion}. This new characterization of right key tableaux, which may be of independent interest, deserves a few words.  The classical definition of the right key rests on an action of the symmetric group on the columns of a tableau, implemented through local jeu de taquin moves.  The other descriptions available in the literature, such as the scanning procedure of \cite{Willis2013}, are algorithmic in nature as well.  We replace these local procedures by a global extremal condition.  To be more precise, we introduce a tableau invariant obtained from suprema of decreasing subwords: to a tableau $T$ we associate the supremum, in a suitable order on decreasing words, of the set of decreasing subwords of the reading word of $T$.  We then prove that this invariant coincides with the classical right key (see Propositions~\ref{IgualdadKeys0} and \ref{IgualdadKeys}).  Suprema of sets of decreasing subwords are easily seen to be invariant under elementary Knuth transformations (Lemma~\ref{SupKnuthEquivalent}), so that our invariant is compatible with Knuth equivalence, and hence with the RSK correspondence, a compatibility that is not apparent from the usual definitions of the right key.  The mechanism at work is the one behind Greene's invariants, and our proof of Knuth invariance is modelled on the classical one.

We should mention when the  alphabet $y=y_1,y_2,\dots$ is left out in \eqref{newversion}, our correspondence then specializes to a new proof, using nothing but ordinary RSK, of the nonsymmetric Cauchy identity \eqref{nonsymCauchy} for Demazure characters and atoms (see Corollary~\ref{coroKeys}).

The tableau expansion for $s_\Lambda(x;0)$ that comes out of this analysis also has consequences that go beyond the proof of \eqref{newversion}.  In the final section, we relate the $m$-symmetric Schur functions at $t=0$ to the almost symmetric Schur functions $\mathfrak{as}_\Lambda(x)$ introduced in \cite{AlmostSym}.  We show that these two bases of $R_m$ are connected by a change-of-basis matrix which is unitriangular with respect to the natural order on key tableaux.  Three consequences follow: a tableau generating function for the almost symmetric Schur functions, the construction of the family dual to them with respect to a natural scalar product on $R_m$, and the corresponding Cauchy identities.  We close with Jacobi-Trudi type determinantal formulas for  $\mathfrak s_\Lambda(x)$, $\mathfrak{as}_\Lambda(x)$ and $\mathfrak{as}^*_\Lambda(x)$ in the dominant case (in the antidominant case for the latter), the main tool being the compatibility, provided by the Schützenberger involution, between left keys and right keys.

The article is organized as follows.  Section~\ref{secprelim} gathers the material that we will need on partitions, symmetric functions, $m$-symmetric functions, the Hecke algebra, key polynomials, and tableaux.  Section~\ref{secdualS} introduces the (dual) $m$-symmetric Schur functions, derives the reproducing kernel and the generalization \eqref{genCauchy} of the Cauchy identity, and identifies the functions $\mathfrak s_\Lambda^*(x)$ with key polynomials.  Section~\ref{seckeys} is devoted to right key tableaux: after recalling the classical definition, we establish our new characterization in terms of suprema of decreasing subwords.  Section~\ref{secRSK} contains the heart of the article, namely the restricted RSK correspondence and the combinatorial proof of \eqref{newversion}.  Finally, Section~\ref{Sec:AlmostSym} deals with the almost symmetric Schur functions and with the determinantal formulas.

\section{Preliminaries} \label{secprelim}

This section gathers the definitions and the notation that will be used throughout the article.  The material on $m$-symmetric functions is taken from \cite{mSym,ConchaLapointe}, to which we refer for the proofs.

\subsection{Partitions and symmetric functions} \label{ssecsym}

A partition $\lambda=(\lambda_1 \geq \lambda_2 \geq \cdots \geq \lambda_k>0)$ is a weakly decreasing sequence of positive integers.  Its degree is $|\lambda|=\lambda_1+\cdots+\lambda_k$ and its length is $\ell(\lambda)=k$.  We represent $\lambda$ by a Young diagram with $\lambda_i$ lattice squares in the $i^{th}$ row, from top to bottom (English notation).  Any lattice square $(i,j)$ in the $i$th row and $j$th column of a Young diagram is called a cell.  The conjugate partition $\lambda'$ is the partition whose diagram is obtained from that of $\lambda$ by interchanging rows and columns.  Given a composition $\pmb a$ and a partition $\lambda$, we let $\pmb a \cup \lambda$ be the partition obtained by reordering the entries of the concatenation of $\pmb a$ and $\lambda$.  Finally, the dominance order $\geq$ on partitions of a same degree is such that $\lambda \geq \mu$ whenever $\lambda_1+\cdots+\lambda_i\geq \mu_1+\cdots+\mu_i$ for all $i$.

Let $\mathbf \Lambda$ be the ring of symmetric functions in the variables $x_1,x_2,x_3,\dots$, and let $h_r$, $e_r$ and $p_r$ stand respectively for the homogeneous, elementary and power-sum symmetric functions of degree $r$.  The Schur function $s_\lambda(x)$ is given by either of the two Jacobi-Trudi determinants
\begin{equation} \label{JacobiTrudiSym}
s_{\lambda}(x)=\det \Bigl( h_{\lambda_i-i+j}(x) \Bigr)_{1 \leq i,j \leq \ell(\lambda)}
=\det \Bigl( e_{\lambda'_i-i+j}(x) \Bigr)_{1 \leq i,j \leq \ell(\lambda')},
\end{equation}
where $h_k(x)=e_k(x)=0$ if $k<0$.

Throughout the article, it will prove convenient to use the plethystic notation in which, for a symmetric function $f$ and $X=x_1+x_2+\cdots$, we let $f[X]=f(x)=f(x_1,x_2,\dots)$.  More generally, we have using this notation that $f[X+x_1+\cdots+x_k]=f(x_1,\dots,x_k,x_1,x_2,\dots)$.

Let $\nu$ be a partition of length $\ell$.  For a sequence of alphabets $X_1,\dots,X_\ell$, where $\ell=\ell(\nu)$, the multi-Schur function (or flagged Schur function)  $s_\nu(X_1,\dots,X_\ell)$ is defined as \cite{mSym}
\begin{equation} \label{eqmulti}
s_\nu(X_1,\dots,X_\ell) = \det \Bigl( h_{\nu_i-i+j}[X_i] \Bigr)_{1 \leq i,j \leq \ell}.
\end{equation}
Observe, from \eqref{JacobiTrudiSym},
that $s_\nu(X_1,\dots,X_\ell)$ is equal to the usual Schur function
$s_\nu(x)$ whenever $X=X_1=X_2=\cdots=X_\ell $.

\subsection{The ring of $m$-symmetric functions} \label{ssecRm}

For a nonnegative integer $m$, the ring $R_m$ of $m$-symmetric functions is the subring of $\mathbb Q(t)[[x_1,x_2,x_3,\dots]]$ made of formal power series that are symmetric in the variables $x_{m+1},x_{m+2},x_{m+3},\dots$, the first $m$ variables thus playing a distinguished, non-symmetric role.  In other words,
$$R_m \simeq \mathbb Q(t)[x_1,\dots,x_m] \otimes \mathbf \Lambda_m,$$
where $\mathbf \Lambda_m$ is the ring of symmetric functions in the variables $x_{m+1},x_{m+2},x_{m+3},\dots$.  It is immediate that $R_0=\mathbf \Lambda$ is the usual ring of symmetric functions and that $R_0 \subseteq R_1 \subseteq R_2 \subseteq \cdots $.

Bases of $R_m$ are naturally indexed by $m$-partitions, which are pairs $\Lambda=(\pmb a;\lambda)$, where $\pmb a= (a_1,\dots,a_m) \in \mathbb Z_{\geq 0}^m$ is a composition with $m$ parts, and where $\lambda$ is a partition.  The entries of $\pmb a$ and $\lambda$ will be called respectively the non-symmetric and the symmetric entries of $\Lambda$.  Unless stated otherwise, $\Lambda$ and $\Omega$ will always stand respectively for the $m$-partitions $\Lambda=(\pmb a;\lambda)$ and $\Omega=(\pmb b;\mu)$.  The degree of $\Lambda$ is $|\Lambda|=a_1+\cdots+a_m+\lambda_1+\lambda_2+\cdots$, while its length is $\ell(\Lambda)=m+\ell(\lambda)$.  We will say that $\pmb a$ is dominant (resp. antidominant) if $a_1 \geq a_2 \geq \cdots \geq a_{m}$ (resp. $a_1 \leq a_2 \leq \cdots \leq a_{m}$), and by extension we will say that $\Lambda=(\pmb a;\lambda)$ is dominant (resp. antidominant) if $\pmb a$ is.  If $\pmb a$ is not dominant, we let $\pmb a^+$ be the dominant composition obtained by reordering its entries.

There is a natural way to represent an $m$-partition by a Young diagram.  The diagram corresponding to $\Lambda$ is the Young diagram of $\pmb a \cup \lambda$ with an $i$-circle added to the right of the row of size $a_i$ for $i=1,\dots,m$ (if there are many rows of size $a_i$, the circles are ordered from top to bottom in increasing order).  For instance, given $\Lambda=(2,0,2,1; 3,2 )$, we have
 $$
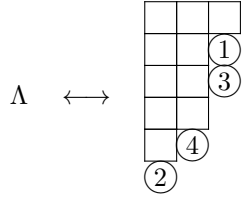
\Lambda \quad \longleftrightarrow  \quad {\tableau[scY]{&& & \bl \\& & \bl \cercle{1} \\& & \bl \cercle{3}\\ & & \bl  \\ & \bl \cercle{4} \\ \bl \cercle{2} }}
$$
Observe that when $m=0$, the diagram associated to $\Lambda=( ; \lambda)$ coincides with the Young diagram associated to $\lambda$.  We let $\Lambda^{(0)}=\pmb a \cup \lambda$, that is, $\Lambda^{(0)}$ is the partition obtained from the diagram of $\Lambda$ by discarding all the circles.  In the example above, $\Lambda^{(0)}=(3,2,2,2,1)$.

Given a composition $\pmb a=(a_1,\dots,a_m)$, we let $|\pmb a|=a_1+\cdots+a_m$ and we let ${\rm Inv}(\pmb a)$ be its number of inversions, that is,
$$
{\rm Inv}(\pmb a)=\#\{(i,j) \, | \, 1\leq i<j\leq m {\rm ~and~} a_{i} < a_j  \}.
$$
Finally, we let $x^{\pmb a}=x_1^{a_1} \cdots x_m^{a_m}$, and we define
\begin{equation} \label{eqk}
k_\Lambda(x) := x^{\pmb a} s_\lambda(x),
\end{equation}
which provides a basis of $R_m$.  It should be observed that the variables in $s_\lambda(x)$ start at $x_1$ instead of $x_{m+1}$.

\subsection{The Hecke algebra and key polynomials} \label{ssecHecke}

All the families of polynomials that we will consider are obtained from a dominant one by acting with divided-difference type operators.  We now introduce those operators.  Let the exchange operator $K_{i,j}$ be such that
$$K_{i,j} f(\dots, x_i,\dots,x_j,\dots)= f(\dots, x_j,\dots,x_i,\dots).$$
We then define the generators $T_i$ of the Hecke algebra $\mathcal H_N$ as
\begin{equation} \label{eqTi}
T_i=t+\frac{tx_i-x_{i+1}}{x_i-x_{i+1}}(K_{i,i+1}-1),\quad i=1,\ldots,N-1.
\end{equation}
More generally, given a reduced decomposition $w=\sigma_{i_1} \cdots \sigma_{i_\ell}$ of the permutation $w \in \mathfrak S_N$, we let $T_w=T_{i_1} \cdots T_{i_\ell}$.  This is well-defined due to the
relations:
\begin{align*} &(T_i-t)(T_i+1)=0\nonumber\\
&T_iT_{i+1}T_i=T_{i+1}T_iT_{i+1}, \quad i=1,\dots,N-2\nonumber\\
&T_iT_j=T_jT_i \, ,\quad |i-j| > 1. 
\end{align*}
The operator $T_j$ is seen to be invertible from the quadratic relation
$(T_i-t)(T_i+1)=0$.  Explicitly, we can check that
$$
\bar T_j :=  T_j^{-1}=t^{-1}-1+t^{-1}T_j.
$$
We will denote by $\omega_m=[m,m-1,\dots,1]$ the longest permutation in ${\mathfrak S}_m$.

The non-symmetric Hall-Littlewood polynomials $H_{\pmb a}(x_1,\dots,x_m;t)$ can be constructed recursively as follows.  If 
$\pmb a$ is dominant then $H_{\pmb a}(x;t)=x^{\pmb a}$.  Otherwise, we define recursively by
$T_i H_{\pmb a}(x;t)=H_{s_i \pmb a}(x;t)$ if $a_i > a_{i+1}$ (with $s_i \pmb a=(a_1\dots,a_{i+1},a_i,\dots,a_m)$).  Combining a non-symmetric Hall-Littlewood polynomial and a Schur function, we define for $\Lambda=(\pmb a;\lambda)$ the function
$$
k_\Lambda(x;t) = H_{\pmb a}(x_1,\dots,x_m;t) s_\lambda(x),
$$
which is the $t$-deformation of the basis \eqref{eqk}.  Since $H_{\pmb a}(x;1)=x^{\pmb a}$, the $k_\Lambda(x;t)$'s still provide a natural basis for the ring $R_m$ of $m$-symmetric functions.

At $t=0$, the Hecke algebra generators degenerate into the operators governing key polynomials.  We have that $\lim_{t \to 0} T_i =\hat \pi_i $, where
$$
\hat \pi_i = \frac{x_{i+1}}{x_i-x_{i+1}} (1-K_{i,i+1}).
$$
We also have that $\lim_{t \to 0} \left(t \bar T_i\right) = \pi_i $, where $\pi_i=\hat \pi_i+1$.  One checks directly that $\hat \pi_i^2=-\hat \pi_i$ and $\pi_i^2=\pi_i$. We define $\pi_{\mathrm{Id}}=\hat{\pi}_{\mathrm{Id}}=\mathrm{Id}$ and, recursively, if
$\sigma'=s_i\sigma$ and $\sigma^{-1}(i)<\sigma^{-1}(i+1)$ (which is to say that $\ell(\sigma')=\ell(\sigma)+1$), then
\[
\pi_{\sigma'}=\pi_i\pi_\sigma,
\qquad
\hat{\pi}_{\sigma'}=\hat{\pi}_i\hat{\pi}_\sigma.
\]

Using these operators, we can define the key polynomials $K_{\pmb a}(x_1,\dots,x_N)$ and dual key polynomials $\hat K_{\pmb a}(x_1,\dots,x_N)$ recursively as:
\begin{equation} \label{defdualkey}
\hat K_{\pmb a}(x) = H_{\pmb a}(x;0)= \left \{ 
\begin{array}{ll}
x^{\pmb a} & {\rm if~} \pmb a {\rm ~is~dominant} \\
\hat \pi_i \hat K_{s_i \pmb a}(x) & {\rm if~} a_i < a_{i+1}
\end{array} \right .,
\end{equation}
and
\begin{equation}
\label{defkey}
 K_{\pmb a}(x) = \left \{ 
\begin{array}{ll}
x^{\pmb a} & {\rm if~} \pmb a {\rm ~is~dominant} \\
 \pi_i  K_{s_i \pmb a}(x) & {\rm if~} a_i < a_{i+1}
\end{array} \right .,
\end{equation}
where  $\pmb a$ is the composition $\pmb a=(a_1,\dots,a_N)$.  Key polynomials were introduced in \cite{Demazure1974a,Demazure1974b}, and are also known as Demazure characters, while the dual key polynomials are also known as Demazure atoms.

\subsection{Tableaux, words and the RSK correspondence} \label{ssectab}

For any totally ordered alphabet $\mathcal B$, let ${\rm Tab}_{\mathcal B}$ be the set of semistandard Young tableaux in the alphabet $\mathcal B$, that is, the set of fillings of a Young diagram with letters of $\mathcal B$ that are weakly increasing along rows (from left to right) and strictly increasing along columns (from top to bottom).  The elements of ${\rm Tab}_{\mathcal B}$ will simply be called tableaux, and we will denote by ${\rm Tab}_{\mathcal B}(\lambda)$ the set of those of shape $\lambda$.  Skew tableaux, that is, fillings of a skew diagram $\lambda/\mu$ obeying the same conditions, will always be qualified as such.  We let $T(i,j)$ stand for the entry of $T$ in cell $(i,j)$, and we let $x^T$ be the monomial in which the power of $x_i$ is the number of occurrences of the letter $i$ in $T$. Two tableaux $P$ and $Q$ of the same shape $\lambda$ are such that $P \leq Q$ iff $P(i,j) \leq Q(i,j)$ for any $(i,j) \in \lambda$.
Given a tableau $T$, we let $w(T)$ be the word obtained by reading the entries of $T$ from left to right and from bottom to top.  We will also let $T_{r}$ be the subtableau of $T$ obtained by considering only the columns of $T$ weakly to the right of column $r$.

Recall that two words $u$ and $w$ are Knuth equivalent, which we denote by $u \equiv w$, if one can be obtained from the other by a sequence of elementary Knuth transformations $(K')$ or $(K'')$:
\begin{align*}
  (K'): \quad y x z& \equiv  y z x, {\rm~for~}x < y \leq z \\
  (K''): \quad x zy & \equiv  z x y, {\rm~for~}x \leq y < z. 
\end{align*}  

\begin{definition}
Let $S$ be a skew semistandard tableau. Its \emph{rectification},
denoted by $\operatorname{rect}(S)$, is the unique semistandard tableau
of straight shape whose reading word is Knuth equivalent to the reading
word of $S$.
\end{definition}
See \cite{Fulton1996} for a detailed treatment.

We finally recall the row insertion algorithm and the RSK correspondence, referring to \cite{Stanley_Fomin_1999,Fulton1996} for the details.  Given a tableau $T$ and a letter $j$, the tableau $T \leftarrow j$ is obtained by inserting $j$ in the first row of $T$, where it bumps the leftmost entry strictly larger than $j$ (if there is no such entry, $j$ is appended at the end of the row and the algorithm stops); the bumped entry is then inserted in the next row in the same fashion, and so on until an entry is appended at the end of a row.  The entries $c_{k-1},\dots,c_1,c_0(=j)$ successively inserted in this process, where $c_{k-1}$ is the entry appended in row $k$, form what we call the insertion path of $T \leftarrow j$.  Note that $w(T \leftarrow j)$ is Knuth equivalent to $w(T)j$.

Given a biword
$$
\left(
\begin{array}{cccc}
i_1 & i_2 & \cdots & i_r \\
j_1 & j_2 & \cdots & j_r
\end{array}  
\right)
$$
in lexicographic order (that is, such that $i_1 \leq i_2 \leq \cdots \leq i_r$, and such that $j_k\leq j_{k+1}$ whenever $i_k = i_{k+1}$), the RSK correspondence builds a pair $(P,Q)$ of tableaux of the same shape by inserting successively the letters $j_1,\dots,j_r$ in $P$, the tableau $Q$ recording in which cell the shape has grown at each step.  It is a bijection between biwords and pairs of tableaux of the same shape.

\section{The (dual) $m$-symmetric Schur functions} \label{secdualS}
We introduce in this section the (dual) $m$-symmetric Schur functions and derive the generalization of the Cauchy identity that the rest of the article will be devoted to proving combinatorially.
We will first define the dual $m$-symmetric Schur functions, and then, by duality, obtain the  
$m$-symmetric Schur functions.  Recall that the multi-Schur functions were introduced in \eqref{eqmulti} and that the Hecke algebra generators $T_i$ were introduced in \eqref{eqTi}.

The dual $m$-symmetric Schur functions are defined first in the dominant case, where they are multi-Schur functions, and then in general by acting with the $T_i$'s.
\begin{definition} \label{defdualS}
  The dual $m$-symmetric Schur functions $ s_{\Lambda}^*(x;t)$
  are defined recursively in the following way.  If $\Lambda=(\pmb a;\lambda)$ is dominant ($\Lambda=(\pmb a;\lambda)$ with  $a_1 \geq a_2 \geq \cdots \geq a_{m}$), then
$$
  s_\Lambda^*(x;t)=s_\nu(X_1,\dots,X_\ell),
$$
where $\nu=\Lambda^{(0)}=\pmb a \cup \lambda$, and where
 $X_i$ stands for the alphabet $X+x_1+\cdots +x_k$ with $k$ the number of circles weakly above row $i$ in the diagram corresponding to $\Lambda$.
 Otherwise, if $a_i<a_{i+1}$  then
 \begin{equation} \label{recursis}
s^*_\Lambda(x;t)= T_i s_{\tilde \Lambda}^*(x;t),
 \end{equation}
where $\tilde \Lambda=s_i \Lambda$, and where $T_i$ is a Hecke algebra generator. This amounts to saying that
\begin{equation} \label{stsigma}
s^*_\Lambda(x;t)=  T_{\sigma^{-1}} s_{\Lambda^+}^*(x;t),
\end{equation}
where $\sigma$ is the shortest permutation such that $\sigma(\pmb a)=(a_{\sigma^{-1}(1)},\dots,a_{\sigma^{-1}(m)}) = \pmb a^+$. 
\end{definition}

\begin{example}
\label{ExamplemSchurdual}
The diagram associated to $(2,1;3,1)$ is 
$$
{\tableau[scY]{&& & \bl \\& & \bl \cercle{\rm 1} \\&  \bl \cercle{\rm 2}\\ & \bl  \\ }}
$$
from which we deduce that
   $$
s_{2,1;3,1}^*(x;t)= \left|
\begin{array}{cccc}
  h_3[X] & h_4[X] & h_5[X] & h_6[X ]\\
   h_1[X+x_1 ]& h_2[X+x_1] & h_3[X+x_1] & h_4[X+x_1] \\
   0 & h_0[X+x_1+x_2]  & h_1 [X+x_1+x_2]& h_2[X+x_1+x_2]\\
   0 & 0 & h_0[X+x_1+x_2] & h_1[X+x_1+x_2]
\end{array}   
\right| .
$$
\end{example}

A bilinear scalar product $\langle \cdot, \cdot \rangle_m$ on $R_m$ is defined by requiring that the $\{k_\Lambda(x;t)\}_\Lambda$ basis be such that
\begin{equation} \label{scalprod1}
\langle k_\Lambda(x;t),  k_\Omega(x;t)  \rangle_m = \delta_{\Lambda \Omega} t^{{\rm Inv} (\pmb a)},
\end{equation}
where we recall that
${\rm Inv} (\pmb a)$
is the number of inversions in $\pmb a$.
It can be shown that the dual $m$-symmetric Schur functions form a basis of $R_m$ \cite{ConchaLapointe}.  The $m$-symmetric Schur functions $s_\Lambda(x;t)$ can thus be defined as the unique basis of $R_m$ such that
\begin{equation} \label{scalprodS}
\langle s_\Lambda(x;t),  s^*_\Omega(x;t)  \rangle_m = \delta_{\Lambda \Omega} t^{{\rm Inv} (\pmb a)}.
\end{equation}

We will now give a reproducing kernel for the scalar product \eqref{scalprod1}.
It was established in \cite{ConchaLapointe} that
 $$
{{t^{-\ell({\mathbf \omega}_m)} }} T^{(y)}_{{\mathbf \omega}_m} \frac{\prod_{i+j \leq m} (1-t x_i y_j)}{\prod_{i+j \leq m+1}(1- x_i y_j)}   =\sum_{\pmb a} t^{-{\rm Inv}(\pmb a)} H_{\pmb a} (x;t) H_{\pmb a} (y;t),
$$
where $\omega_m=[m,m-1,\dots,1]$ is the longest permutation in ${\mathfrak S}_m$, and where $ T^{(y)}_{{\mathbf \omega}_m} $ stands for $ T_{{\mathbf \omega}_m}$
acting on the variables $y_1,\dots,y_m$. Adding the Cauchy kernel 
$$
\frac{1}{\prod_{i,j} (1-x_i y_j)} = \sum_\lambda s_\lambda(x) s_\lambda(y)
$$
to the previous equation then leads, using $ k_\Lambda(x;t)= H_{\pmb a} (x;t) s_\lambda(x)$, to
 \begin{equation} \label{reprokernel}
{{t^{-\ell({\mathbf \omega}_m)} }} T^{(y)}_{{\mathbf \omega}_m} \frac{\prod_{i+j \leq m} (1-t x_i y_j)}{\prod_{i+j \leq m+1}(1- x_i y_j)} \left[ \frac{1}{\prod_{i,j} (1-x_i y_j)} \right]
=\sum_{\Lambda} t^{-{\rm Inv}(\pmb a)} k_\Lambda(x;t)  k_\Lambda(y;t),
\end{equation}
where the new product can be located to the right given that it commutes,  by symmetry in the $y$ variables, with  
$ T^{(y)}_{{\mathbf \omega}_m}$.  Comparing with \eqref{scalprod1},
we see immediately that \eqref{reprokernel} is the reproducing kernel that we were seeking.  As such, we have by \eqref{scalprodS} that
 \begin{equation*} 
{{t^{-\ell({\mathbf \omega}_m)} }} T^{(y)}_{{\mathbf \omega}_m} \frac{\prod_{i+j \leq m} (1-t x_i y_j)}{\prod_{i+j \leq m+1}(1- x_i y_j)} \left[ \frac{1}{\prod_{i,j} (1-x_i y_j)} \right]
=\sum_{\Lambda} t^{-{\rm Inv}(\pmb a)} s_\Lambda(x;t)  s_\Lambda^*(y;t).
\end{equation*}
 To consider the limit $t\to 0$, it proves convenient to rewrite the previous equation as
 \begin{equation*} 
      \frac{\prod_{i+j \leq m} (1-t x_i y_j)}{\prod_{i+j \leq m+1}(1- x_i y_j)} \left[ \frac{1}{\prod_{i,j} (1-x_i y_j)} \right]
=\sum_{\Lambda}  s_\Lambda(x;t)  \left( {{t^{\ell({\mathbf \omega}_m) -{\rm Inv}(\pmb a)} }} \bar T^{(y)}_{{\mathbf \omega}_m}  s_\Lambda^*(y;t) \right),
\end{equation*}
where $\bar T_{{\mathbf \omega}_m} =  ( T_{{\mathbf \omega}_m})^{-1}$. Taking the limit  $t\to 0$, we then obtain the following generalization of the Cauchy identity
\begin{equation} \label{genCauchy}
\frac{1}{\left[\prod_{i+j \leq m+1}(1- x_i y_j) \right] \left[\prod_{i,j} (1- x_i y_j) \right] }   =\sum_{\Lambda} s_\Lambda(x;0) \,  {\mathfrak s}_\Lambda^*(y), 
\end{equation}
where
\begin{equation} \label{dualschur0}
\mathfrak s_{\Lambda}^*(x) := \lim_{t \to 0} t^{\ell({\mathbf \omega}_m)-{\rm Inv}(\pmb a)} \bar  T_{{\mathbf \omega}_m}  s_{\Lambda}^* (x;t). 
\end{equation}
We will refer to these functions  as the dual $m$-symmetric Schur functions at $t=0$. 

The goal of this article is to provide a combinatorial proof of the identity \eqref{genCauchy}.  As we will soon see, the functions $\mathfrak s_{\Lambda}^*(x)$ will turn out to be key polynomials.  As such, they will have a natural combinatorial interpretation. On the other hand, we will not be able to give directly a combinatorial interpretation for the $m$-symmetric Schur functions at $t=0$, $s_\Lambda(x;0)$. Remarkably, such a combinatorial interpretation will emerge as a corollary of our combinatorial proof of \eqref{genCauchy}.

We close this section by making explicit the connection between the functions $\mathfrak s_{\Lambda}^*(x)$ and the key polynomials \eqref{defkey}.  When $\Lambda$ is dominant, $s_\Lambda^*(x;t)$ is a flagged Schur function, and is therefore a key polynomial (see \cite{FlaggedAreKey1995}, theorem 23).
\begin{proposition} \label{PropKey}
  Suppose that $\Lambda=(\pmb a; \lambda)$ is dominant and that $\lambda$ is of length $\ell$.
  For any $N \geq \ell(\lambda)$ we have that
\begin{equation}  
  s_\Lambda^*(x_1,\dots,x_m,x_{m+1},\dots,x_N;t) = K_{(0^{N-\ell},\lambda_\ell,\dots,\lambda_1,a_1,\dots,a_m)}(x_{1},\dots,x_N,x_1,\dots,x_m).
\end{equation}  
\end{proposition}  
In the general case, we are interested in the functions
$ \mathfrak s_{\Lambda}^*(x)$.
We thus have that if $\sigma (\pmb a) = (a_{\sigma^{-1}(1)},\dots, a_{\sigma^{-1}(m)})=\pmb a^+$, then
$$
\mathfrak s_{\Lambda}^*(x) 
= \lim_{t \to 0} t^{\ell({\mathbf \omega}_m)-{\rm Inv}(\pmb a)} \bar  T_{{\mathbf \omega}_m}  T_{\sigma^{-1}} s_{\Lambda^+}^*(x;t) 
= \lim_{t \to 0} t^{\ell({\mathbf \omega}_m \sigma^{-1})} \bar  T_{{\mathbf \omega}_m \sigma^{-1}} s_{\Lambda^+}^*(x;t)    =  \pi_{{\mathbf \omega}_m \sigma^{-1}}  s_{\Lambda^+}^*(x;t).  
$$
We thus get from Proposition~\ref{PropKey} that the functions $\mathfrak s_{\Lambda}^*(x)$ are still key polynomials when the number of variables is finite.
The next proposition will give this relation explicitly.
\begin{proposition} \label{propKeyS}
  Suppose that $\Lambda=(\pmb a; \lambda)$ (not necessarily dominant) is  such that   $\lambda=(\lambda_1,\dots,\lambda_\ell)$.  For any $N \geq \ell(\lambda)$ we have that
\begin{equation}  
\mathfrak s_{\Lambda}^*(x_1,\dots,x_N)
  = K_{(0^{N-\ell},\lambda_\ell,\dots,\lambda_1,a_m,\dots,a_1)}(x_{1},\dots,x_N,x_1,\dots,x_m).
\end{equation}
\end{proposition}  
\begin{proof} Let $\sigma$ be such that
$\sigma (\pmb a) = (a_{\sigma^{-1}(1)},\dots, a_{\sigma^{-1}(m)})=\pmb a^+$.  We have just seen that
$
\mathfrak s_{\Lambda}^*(x) =  \pi_{{\mathbf \omega}_m \sigma^{-1}}  s_{\Lambda^+}^*(x;t)   
$, and Proposition~\ref{PropKey} therefore gives
$$
\mathfrak s_{\Lambda}^*(x_1,\dots,x_N) =  \pi_{{\mathbf \omega}_m \sigma^{-1}} K_{(0^{N-\ell},\lambda_\ell,\dots,\lambda_1,a_{\sigma^{-1}(1)},\dots, a_{\sigma^{-1}(m)})}(x_{1},\dots,x_N,x_1,\dots,x_m). 
$$
The sequence $(0^{N-\ell},\lambda_\ell,\dots,\lambda_1)$ being weakly increasing, we have that
\begin{align*}
  & K_{(0^{N-\ell},\lambda_\ell,\dots,\lambda_1,a_{\sigma^{-1}(1)},\dots, a_{\sigma^{-1}(m)})}(x_{w(1)},\dots,x_{w(N)},x_1,\dots,x_m) \\
  & \qquad \qquad \qquad \qquad= K_{(0^{N-\ell},\lambda_\ell,\dots,\lambda_1,a_{\sigma^{-1}(1)},\dots, a_{\sigma^{-1}(m)})}(x_{1},\dots,x_N,x_1,\dots,x_m) 
\end{align*}
for any $w \in \mathfrak S_N$. The operator $\pi_{{\mathbf \omega}_m \sigma^{-1}}$ thus only acts on the last $m$ entries of the Key polynomial and we get
\begin{align*}
   \mathfrak s_{\Lambda}^*(x_1,\dots,x_N)&= \pi_{{\mathbf \omega}_m \sigma^{-1}}  K_{(0^{N-\ell},\lambda_\ell,\dots,\lambda_1,a_{\sigma^{-1}(1)},\dots, a_{\sigma^{-1}(m)})}(x_{1},\dots,x_N,x_1,\dots,x_m) \\
   &  = K_{(0^{N-\ell},\lambda_\ell,\dots,\lambda_1,a_{\omega_m(1)},\dots, a_{\omega_m(m)})}(x_{1},\dots,x_N,x_1,\dots,x_m) \\
    &  = K_{(0^{N-\ell},\lambda_\ell,\dots,\lambda_1,a_m,\dots, a_1)}(x_{1},\dots,x_N,x_1,\dots,x_m).
\end{align*} 
\end{proof}

\section{Right key tableaux and their different characterizations} \label{seckeys}

Right key tableaux are the combinatorial invariant that governs the tableau expansion of the key polynomials, and they will accordingly govern the admissibility conditions of the next section.  We recall their classical definition in Section~\ref{sseckeyclassical}, and then establish in Section~\ref{sseckeynew} the new characterization on which the rest of the article rests.

\subsection{The classical characterization} \label{sseckeyclassical}

A combinatorial formula for the key polynomials \eqref{defkey} in terms of key tableaux was obtained in \cite{Lascoux1990a}, and it is this characterization that we will use throughout the article.  Unless stated otherwise, all the tableaux appearing in this section are taken in the alphabet $\mathbb N=\{1,2,3,\dots\}$.

\begin{definition}
    Let $T$ be a tableau with $\ell$ columns, and let
$C_1,C_2,\ldots,C_\ell$ denote its columns from left to right.
We say that $T$ is a \emph{key tableau} if
\[
C_i \supseteq C_j
\qquad\text{whenever }1\le i<j\le \ell .
\]
\end{definition}

\begin{example} The tableau 
\begin{ytableau}
1 & 2 & 2 \\
2 & 3 \\
3
\end{ytableau}
is a key tableau because $\{1,2,3 \} \supseteq \{ 2,3 \} \supseteq \{2\}$.
\end{example}

\begin{example} The tableau
\begin{ytableau}
1 & 3 \\
2
\end{ytableau}
is not a key tableau because $\{ 1,2\}$ does not contain $\{ 3 \}$.
\end{example}

Given a tableau $T$ of shape $\lambda$ with column lengths $\pmb c= ( c_{1}, c_{2} , \ldots , c_{\ell} )$, for each permutation $\sigma$ we are interested in the skew tableau $S$ with column lengths $\sigma \pmb c=( c_{\sigma^{-1}(1)}, c_{\sigma^{-1}(2)} , \ldots , c_{\sigma^{-1}(\ell)} )$  whose rectification is $T$.
In appendix A5 of \cite{Fulton1996} it is shown that one can find $S$ using an action of the symmetric group on the columns of $T$.
We will use the jeu de taquin construction found in appendix A1.2 of \cite{Stanley_Fomin_1999}.
\begin{example}
\label{ExampleActionColumns}
The action of the symmetric group on consecutive columns is as follows.
    We start by adding a hole at the bottom of the column to the right.  We then
    use jeu de taquin slides to move the hole to the top of the column to the left as follows:
$$
\tableau[scY]{
1 & 3 \\
4 & 7 \\
5 & 8 \\
6  \\
7  \\
9
}
\qquad
\tableau[scY]{
1 & 3 \\
4 & 7 \\
5 & 8 \\
6 & \bl \times \\
7  \\
9
}
\qquad
\tableau[scY]{
1 & 3 \\
4 & 7 \\
5 & \bl \times \\
6 & 8 \\
7  \\
9
}
\qquad
\tableau[scY]{
1 & 3 \\
4 & \bl \times \\
5 & 7 \\
6 & 8 \\
7  \\
9
}
\qquad
\tableau[scY]{
1 & 3 \\
\bl \times & 4 \\
5 & 7 \\
6 & 8 \\
7  \\
9
}
\qquad
\tableau[scY]{
\bl \times & 3 \\
1 & 4 \\
5 & 7 \\
6 & 8 \\
7  \\
9
}
\qquad
\tableau[scY]{
\bl & 3 \\
1 & 4 \\
5 & 7 \\
6 & 8 \\
7  \\
9
}
$$
We repeat this procedure until the two column lengths have been exchanged. In this example, this corresponds to  the  following steps:
$$
\tableau[scY]{
1 & 3 \\
4 & 7 \\
5 & 8 \\
6 \\
7  \\
9
}
\quad
\longrightarrow
\quad
\tableau[scY]{
\bl & 3 \\
1 & 4 \\
5 & 7 \\
6 & 8 \\
7  \\
9
}
\quad
\longrightarrow
\quad
\tableau[scY]{
\bl & 3 \\
\bl & 4 \\
1 & 5 \\
6 & 7 \\
7 & 8 \\
9
}
\quad
\longrightarrow
\quad
\tableau[scY]{
\bl & 3 \\
\bl & 4 \\
\bl & 5 \\
1 & 7 \\
6 & 8 \\
7 & 9
}
$$
\end{example}    
This defines an action of $\mathfrak{S}_{\ell}$ on skew tableaux.  Every skew tableau produced by this action has an equivalent representative in which any two adjacent columns either start or end on the same row; such a representative is called the compact form of $\sigma(T)$, and we may restrict our attention to those.  The next example, taken from \cite{Fulton1996}, illustrates the distinction.
\begin{example}
Given the tableau
$$
T =
\tableau[scY]{
1 & 1 & 2 & 2 \\
2 & 3 & 3 \\
4
}
$$
both skew tableaux $S$ and $S'$ below have column lengths $(2,3,2,1)$ and rectify to $T$:
$$
S=
\tableau[scY]{
\bl & 1 & 2 & 2 \\
1 & 3 & 3 \\
2 & 4
}
\qquad
S'=
\tableau[scY]{
\bl & \bl & \bl & 2 \\
\bl &  1 & 2 \\
\bl &  3 & 3 \\
1 & 4 \\
2
}
$$
Only $S$, however, is in compact form.
\end{example}

\begin{example}
\label{EjAccionT}
Let
$$
T = \tableau[scY]{
1 & 3 & 4 \\
2 & 6 \\
5
}
$$
The complete orbit of $T$ under the action of $\mathfrak S_3$ is
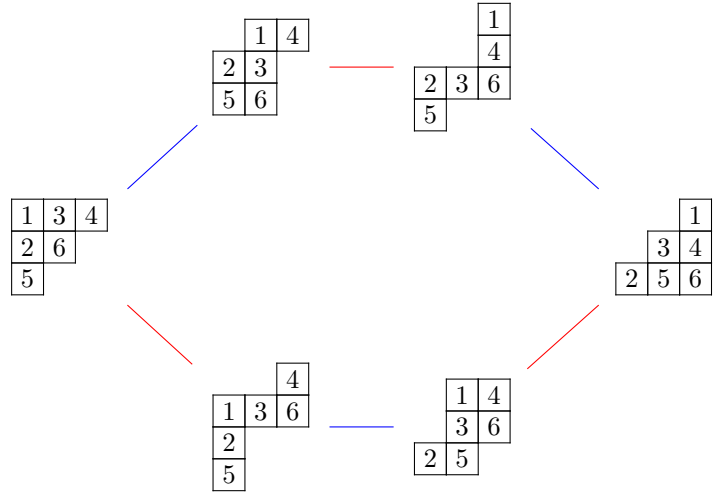
\begin{center}
\begin{tikzcd}
& \usebox\yTTBAC  \arrow[r, red, dash] & \usebox\yTTCAB \arrow[rd, blue, dash] & \\
\usebox\yTTABC \arrow[ru, blue, dash] \arrow[rd, red, dash]& &   & \usebox\yTTCBA \\
& \usebox\yTTACB \arrow[r, blue, dash] & \usebox\yTTBCA \arrow[ru, red, dash] &
\end{tikzcd}
\end{center}
where the blue lines denote the action that swaps the first and second columns, and the red lines denote the action that swaps the second and third columns.
\end{example}

The construction has the following fundamental property \cite{Fulton1996}, which is what makes the notion of key well defined.
\begin{proposition}
\label{CorKeyRep}
Let $T$ be a tableau with $\ell$ columns and let $\mathfrak{S}_{\ell}(T)$ be the set of tableaux obtained from  the action of  $\mathfrak{S}_{\ell}$ on tableaux described above.
If $T' \in \mathfrak{S}_{\ell}(T)$ and $C$ is the rightmost (leftmost) column of $T'$, then $C$ is determined by $T$ and by the length of $C$ alone.
\end{proposition}
It is thus natural to define  $\mathcal{R}_{c}(T)$ (resp. $\mathcal{L}_{c}(T)$) as the
rightmost (resp. leftmost) column of   $T' \in \mathfrak{S}_{k}(T)$ whenever such column is of length $c$.
\begin{proposition}
\label{CorKeyInc}
    If $c<d$, then $\mathcal{R}_{c}(T) \subset \mathcal{R}_{d}(T)$ and $\mathcal{L}_{c}(T) \subset \mathcal{L}_{d}(T)$.
\end{proposition}
We are now in a position to define left and right keys.
\begin{definition}
  Given a tableau $T$ with columns of length $(c_1,\dots,c_\ell)$, 
  its right key tableau $K_{+} (T)$ (resp. left key tableau $K_{-} (T)$)
  is the tableau whose columns are $\mathcal{R}_{c_1}(T), \dots, \mathcal{R}_{c_\ell}(T)$ (resp. $\mathcal{L}_{c_1}(T), \dots, \mathcal{L}_{c_\ell}(T)$)
\end{definition}

\begin{example}
\label{ExampleLeftRight}
  We see from Example~\ref{EjAccionT} that the right and left key tableaux of 
$$
T = \tableau[scY]{
1 & 3 & 4 \\
2 & 6 \\
5
}
$$
are respectively
$$
K_{+} (T) = \tableau[scY]{
1 & 4 & 4 \\
4 & 6 \\
6
}
\qquad    
    {\rm and}
\qquad    
K_{-}(T) = \tableau[scY]{
1 & 2 & 2 \\
2 & 5 \\
5
}
$$
\end{example}

\subsection{A new characterization of right key tableaux} \label{sseckeynew}

We derive in this subsection a new characterization of right key tableaux, based on the supremum of a certain set of words.

We will use the ordered sets
$\mathcal A= \{1,2,\dots,N\} \cup  \{\hat 1,\dots ,\hat m\}$ and  $\hat {\mathcal A}= \{\hat 1,\dots ,\hat m\}$, with the order
$$
1 < 2 < \cdots < N < \hat 1 < \hat 2 < \cdots < \hat m.
$$
We extend the operations of addition and subtraction to the symbols
$\hat{\mathcal A}$ by treating them formally as integers. For example,
$\hat3-\hat1=\hat2$ and
$\hat2+\hat4=\hat6$ whenever $\hat6\in\hat{\mathcal A}$.

\begin{definition}
\label{DefOrderWords}
A word $w=w_1 \cdots w_t$
is said to be decreasing if $w_1 > w_2 > \cdots > w_t$.
Let $W( \mathcal{A} )$ be the set of \textit{decreasing} words on $\mathcal{A}$. Given $v,w \in W ( \mathcal{A} )$, with $v= v_{1} \ldots v_{s} , w = w_{1} \ldots w_{t}$, we say that $v \leq w$ iff $s \leq t$ and $v_{i} \leq  w_{i}$ for $1 \leq i \leq s$. We write $v\subseteq w$ to indicate that $v$ is a subword of $w$.
For a set of decreasing words $A$, we define $\sup A = w_{1} \cdots w_{t}$ where
$$
w_{i} = \max \{ b_{i} | b_{1} \cdots b_{j} \in A, j \geq i\}.
$$
\end{definition}
We first prove three simple propositions.

\begin{proposition}
    Let $A$ be a set of words in $W( \mathcal{A} )$. Then $\sup A \in W(\mathcal{A})$, that is, $\sup A$ is a decreasing word.
\end{proposition}
\begin{proof}
  Let $\sup A = w_{1}\cdots w_{k}$ and argue by induction on the length of the subword, a one-letter word being trivially decreasing.  Suppose then that $w_1 \cdots w_{j-1}$ is decreasing for some $j \leq k$.  By the very definition of  $\sup A$,  no word in $A$ has its $(j-1)$-th letter larger than $w_{j-1}$.  Since $w_{j} = b_{j}$ for some word $b_1 \cdots b_j \cdots b_t  \in A$, we conclude that $w_{j-1} \geq b_{j-1} > b_{j} =w_{j}$, so that $w_1 \cdots w_j$ is decreasing as well.
\end{proof}

\begin{remark}
    The word $\sup A$ defined above is indeed the supremum of $A$ with
    respect to the partial order on $W(\mathcal{A})$.
\end{remark}
\begin{proposition}
  \label{IncImpLeq} If $v$ and $w$ are two decreasing words such that
   $v \subseteq w$, then $v \leq w$.
\end{proposition}

\begin{proof}
  Let $w=w_{1} \dots w_{s}$.  Since $v \subseteq w$, we have that
  $v=w_{i_{1}} \dots w_{i_{r}}$ for some $1 \leq i_{1}<i_{2}  <\ldots  <i_{r}\leq s$.  Hence $j \leq i_{j}$ for all $1 \leq j \leq r$, which implies that
    $w_{j} \geq w_{i_{j}} = v_{j}$ given that $w$ is decreasing.
\end{proof}

\begin{proposition} \label{propaddone}
  Suppose that $v$ and $w$ are two distinct decreasing words such that
  $v \subseteq w$.  If $i > v_1$ and $i > w_1$ ($v_1$ and $w_1$ are the largest letters in $v$ and $w$ respectively), then
  $$
  \sup \{ iv, w \} = iu,
$$
where $u$ is the subword of $w$ obtained by deleting the highest letter of $w$ not in $v$. 
\end{proposition}
\begin{proof} Suppose that $w_{j+1}$ is the largest letter of $w$ not in $v$, and let  $v=v_1 \cdots v_j v'$ and $w=v_1 \cdots v_j w_{j+1} w'$. This gives us that $u=v_1 \cdots v_j w'$.
 We now need to prove that
  $$
   \sup \{ iv, w \} = i u.
  $$
  Set $s = \sup \{ iv, w \}$.
   From $v \subseteq w$ we get $v' \subseteq w'$, and hence $v' \leq w'$ by the previous proposition; therefore $iv \leq iu$.  We also have $w \leq iu$, since $v_1 \cdots v_j w_{j+1} \leq i v_1 \cdots v_j$, and the two inequalities together give $s \leq i u$.  For the reverse inequality it suffices to observe that any decreasing word bounding both $iv$ and $w$ from above must also bound $iu$, which is checked directly on the letters.
\end{proof}

Given a word $w$, we let $W(w)$ be the set of decreasing subwords of $w$.
From now on, a tableau $T$ will always be such that either $T \in {\rm Tab}_{\mathcal A}$ or  $T \in {\rm Tab}_{\{1,\dots,N \}}$.  Recalling that $w(T)$ stands for the reading word of $T$ and that $T_r$ stands for the subtableau made of the columns of $T$ weakly to the right of column $r$, we let $W(T)$ be the set of decreasing subwords of $w(T)$, that is, $W(T)=W(w(T))$.  Finally, if $T \in {\rm Tab}_{\mathcal A}$, we let  $\hat W(T)$ be the set of decreasing subwords of $w(T)$ that only have letters in $\hat{\mathcal A}$.

\begin{example}
Let 
$$
T = \tableau[scY]{
1 & 3 & 3 & \hat{2} \\
2 & \hat{1} & \hat{3} \\
}
$$
Then $w(T) = 2 \hat{1} \hat{3} 1 3 3 \hat{2}$, and
$$
W(T) = \{ \hat{3} \hat{2}, \hat{3}3 , \hat{3}1 , \hat{1} 3,  \hat{1}1, 21 , \hat{3}, \hat{2}, \hat{1}, 3,2,1 \},
$$
$$
\hat{W} (T) = \{ \hat{3} \hat{2}  ,  \hat{3}, \hat{2}, \hat{1} \}.
$$
\end{example}

The whole point of taking a supremum over decreasing subwords is that the result only depends on the Knuth class of the word.  This is the content of the next lemma, in which $(K')$ and $(K'')$ are the elementary Knuth transformations introduced in Section~\ref{ssectab}.
\begin{lemma}
  \label{SupKnuthEquivalent}
If $w_1  \equiv w_2$, then
$$
\sup W(w_1) = \sup W(w_2).
$$
\end{lemma}
\begin{proof}
It suffices to prove the lemma when $w_1$ and $w_2$ differ by a single elementary Knuth transformation. In both cases we can write  $w_{1} = u \cdot \alpha \cdot v$ and $w_{2} = u \cdot \beta \cdot v$, where $\alpha$ and $\beta$ are the two sides of the transformation.  One checks at once that every decreasing subword of $\alpha$ is also a decreasing subword of $\beta$, and that the only decreasing subword of $\beta$ which is not one of $\alpha$ is $zx$.  Consequently, any $w \in W(w_1)$ belongs to $W(w_2)$, which gives in both cases that
$$
\sup W(w_{1}) \leq \sup W(w_{2}).
$$
Only the reverse inequality needs an argument, and only for those $w \in W(w_2)$ that do contain the subword $zx$.  We write any such word as $w= u_{1}zx v_{1}$, with $u_{1}$ a subword of $u$ and $v_{1}$ a subword of $v$.

Consider first the transformation $(K')$, that is, 
$w_{1} = u \cdot y x z \cdot v$ and $w_{2} = u \cdot y z x \cdot v$ with $x < y \leq z$.  Here $u_{1}yxv_{1}$ and $u_{1}z$ both belong to $W(w_{1})$, and $\sup \{ u_{1}yxv_{1},u_{1}z \} = w$. Hence $w  \leq \sup W(w_{1})$, from which we get
$$
\sup W(w_{2}) \leq \sup W(w_{1}).
$$

Consider now the transformation $(K'')$, that is,  $w_{1} = u \cdot xzy \cdot v$ and $w_{2} = u \cdot zxy \cdot v$ with $x \leq y < z$.  This time $u_{1} zy v_{1}$ belongs to $W(w_{1})$ and dominates $w$, so that again
$$
\sup W(w_{2})  \leq \sup W(w_{1}) ,
$$
which concludes the proof.
\end{proof}

\begin{remark}
    The proof is inspired by the classical argument proving the Knuth invariance of Greene's invariants (see \cite{Fulton1996}, Chapter 2).
\end{remark}

In Proposition~\ref{CorKeyInc}, ${\mathcal R}_c(T)$ corresponds to the set of all entries in any column of length $c$ of $K_{+}(T)$. It will be more convenient for our  purposes to use decreasing words instead of sets, and to index the words according  to the columns of $K_+(T)$ instead of to their lengths.
\begin{definition}  Suppose that $T$ has $l$ columns. For $i=1,\dots,l$,
we let ${\mathcal C}_i(T)$ be the decreasing word whose letters are the entries in column $i$ of $K_+(T)$.
\end{definition}

\begin{example}
\label{ExampleKeyCols}
 Let $$
T = \tableau[scY]{
1 & 3 & 4 \\
2 & 6 \\
5
}
$$ as in example \ref{ExampleLeftRight}. Then
$$
K_{+} (T) = \tableau[scY]{
1 & 4 & 4 \\
4 & 6 \\
6
},
$$ hence ${\mathcal C}_{1}(T)=641,\: {\mathcal C}_{2}(T)=64$ and ${\mathcal C}_{3}(T)=4$.
\end{example}

\begin{proposition}
\label{IgualdadKeys0}
For any tableau $T$, we have that 
$$
\sup W (T) = \sup K_{+}(T) = {\mathcal C}_1(T).
$$
\end{proposition}

\begin{proof}
The second equality is immediate: every column of $K_{+}(T)$ is contained in its first column, so that $\sup K_{+}(T) = {\mathcal C}_1(T)$.

We now turn to the first one.  Let $c_1,\dots,c_l$ be the lengths of the columns of $T$ (in weakly decreasing order).  Consider the unique skew tableau $S \in \mathfrak{S}_{\ell}(T)$ in compact form whose column lengths are  $c_l,\dots,c_1$ (that is, in weakly increasing order). Since $S$ rectifies to $T$, we have that $w(T) \equiv w(S)$.  By Lemma~\ref{SupKnuthEquivalent}, we thus have that
$\sup W (T)=\sup W (S)$.  Therefore, we only have left to prove that 
${\mathcal C}_1(T) = \sup W (S)$. Now, by definition of $K_{+}(T)$, the rightmost column of $S$ is ${\mathcal R}_{c_1}(T)$.
But since $S$  is in compact form and $c_l,\dots,c_1$ is a weakly increasing sequence, we have that the bottom entry in each column of $S$ lies in the same row as the bottom entry of its rightmost column ${\mathcal R}_{c_1}(T)$ (considered as a column instead of as a set), that is, $S$ is a counter-tableau (refer to Example~\ref{EjAccionT}). Since the rows are weakly increasing in $S$, it is then obvious that $\sup  {\mathcal R}_{c_1}(T)= \sup W (S)$.
The proposition follows immediately given that  ${\mathcal C}_1(T) = \sup  {\mathcal R}_{c_1}(T)$.
\end{proof}

Let $c_1,\dots,c_l$ be the lengths of the columns of $T$ (in weakly decreasing order). By definition of $K_+(T)$, the first column of $T$ does not play any role in how the right keys  ${\mathcal R}_{c_2}(T),\dots,{\mathcal R}_{c_\ell}(T)$ are constructed. Hence, we also have that $\sup W(T_{2}) = {\mathcal C}_2(T)$. Generalizing this idea, we get the following  proposition. 
\begin{proposition}
\label{IgualdadKeys}
    Let $T$ be a tableau with $l$ columns.  We have that
$$
    \sup W(T_{k}) = {\mathcal C}_k(T)  \quad \text{for every } 1 \leq k \leq l .
$$
\end{proposition}

\section[Key tableaux and $m$-symmetric Schur functions at $t=0$]{Key tableaux and (dual) $m$-symmetric Schur functions at $t=0$.} \label{secRSK}

This section contains the heart of the article.  We first attach to every $m$-partition $\Lambda$ two key tableaux, $K_+^*(\Lambda)$ and $K_+(\Lambda)$, exchanged by an order-reversing involution on the alphabet; the first one describes the dual $m$-symmetric Schur functions at $t=0$, while the second one lets us define the functions $\mathfrak s_\Lambda(x)$ as tableau generating functions.  We then single out the pairs of tableaux that these two descriptions pair off, the admissible pairs, and prove that the ordinary RSK correspondence restricts to a bijection between admissible biwords and admissible pairs.  The Cauchy identity \eqref{newversion}, and with it the identification $\mathfrak s_\Lambda(x)=s_\Lambda(x;0)$, follows at once.

Given a composition $\alpha=(\alpha_1,\dots,\alpha_N)$, let $\sigma$
be the shortest permutation in $\mathfrak S_N$ such that
$$
\alpha^+= (\alpha_{\sigma(1)},\dots,\alpha_{\sigma(N)})
$$
is dominant. We can associate a right key tableau $K_+(\alpha)$ of shape $\alpha^+$ in the following way: if column $k$ of $K_+(\alpha)$ is of length $r$, then it is filled with the entries $\sigma(1),\dots,\sigma(r)$.

We say that two tableaux $P$ and $Q$ of the same shape $\lambda$  are such that $P \leq Q$ iff $P(i,j) \leq Q(i,j)$ for any $(i,j) \in \lambda$. 
It is known (see \cite{Lascoux1990a}) that the key polynomials have the following expansion:
\begin{equation} \label{keysum}
K_\alpha(x_1,\dots,x_N) = \sum_{T} x^T,
\end{equation}
where the sum is over all tableaux $T \in {\rm Tab}_{\{1,\dots,N\}}$ of shape $\alpha^+$ such that $K_+(T) \leq K_+(\alpha)$.  In view of Proposition~\ref{propKeyS}, we have that
the dual $m$-symmetric Schur functions at $t=0$ are given by
\begin{equation} \label{eq1.4.2}
  \mathfrak s_{\Lambda}^*(x_1,\dots,x_N)
   = K_{(0^{N-\ell},\lambda_\ell,\dots,\lambda_1,a_m,\dots,a_1)}(x_{1},\dots,x_N,\hat x_1,\dots,\hat x_m) \big |_{\hat x_1 = x_1,\dots, \hat x_m = x_m }.
\end{equation}
It is thus natural to define, for $\Lambda=(a_1,\dots,a_m;\lambda)$,
the key tableau
$$
K_+^*(\Lambda)= K_+(\gamma),
$$
where $\gamma=(0^{N-\ell},\lambda_\ell,\dots,\lambda_1,a_m,\dots,a_1)$.
In this case, $\gamma^+=\Lambda^{(0)}$, which is the partition corresponding to the diagram of $\Lambda$ without its circles.  
Observe that in $\gamma^+=(\gamma_{\sigma(1)},\dots,\gamma_{\sigma(N+m)})$, we have that $\sigma(i)=N+1-i$ if $\gamma_{\sigma(i)}=\lambda_i$.
 Similarly, we have that
 $\sigma(i)=N+m+1-i$ if
  $\gamma_{\sigma(i)}=a_i
$. 
Hence, the key tableau $K_+^*(\Lambda)$ can be easily extracted from the $m$-partition $\Lambda$.
\begin{lemma}
\label{KeyLambdaDual} The key tableau
$K_{+}^{*} (\Lambda)$ is the unique tableau such that, for $l =1,\dots,m$, every column to the left of the circle filled with an $m+1-l$ contains a letter $\hat l \in \hat {\mathcal A}$ and such that the remaining cells (supposing that there are $k$ of them in a given column) are filled with the letters $N,N-1,\dots,N-k+1$ (if $N-k+1 \leq 0$, then the key tableau $K_{+}^{*} (\Lambda)$ does not exist for that $N$).  
\end{lemma}
\begin{example}
        Let $m=4$ and
        $$
 \Lambda  =   \tableau[scY]{  &  &  &   &  &   &   \bl \tcercle{3} \\ & &  & &  \bl \tcercle{1} \\ & & \\ & &  \bl \tcercle{4} \\  
 \\  \bl \tcercle{2}
}
$$
 Since the circle filled with a $4=4+1-1$ lies in column $3$, the first two columns contain the letter $\hat{1}$. Likewise, the circle filled with a $3=4+1-2$ lies in column $7$, the first six columns contain the letter $\hat{2}$. Continuing in this manner for the letters $\hat{3}$ and $\hat{4}$, we get after filling the remaining cells with the letters $N$ and $\bar N=N-1$ that
$$
K_{+}^{*} ( \Lambda )  =   \tableau[scY]{ {\bar N} & {N} & {N} &  \hat{2} & \hat{2} & \hat{2} \\ {N} & \hat{1} & \hat{2}  & \hat{4} \\ \hat{1} & \hat{2} & \hat{4} \\ \hat{2} & \hat{4} \\  
 \hat{4} 
}
$$
\end{example}
After defining
\begin{equation} \label{defS*}
\mathcal{S}^{*} (\Lambda) =    \{ T \in {\rm Tab}_{\mathcal A}(\Lambda^{(0)}) \,  : \, K_+(T) \leq K_{+}^{*} (\Lambda) \}, 
\end{equation}
we thus obtain  from \eqref{keysum} and \eqref{eq1.4.2} that  the dual $m$-symmetric Schur function at $t=0$  is given by
\begin{equation} \label{dualSchur}
 \mathfrak s_{\Lambda}^*(x_1,\dots,x_N) = \sum_{T \in \mathcal{S}^{*} (\Lambda) } x^T \big |_{\hat x_1 = x_1,\dots, \hat x_m = x_m }. 
\end{equation}
Here the entry $\hat{i}$ contributes a factor $\hat{x}_{i}$ to the monomial $x^{T}$.
 It is understood that if $N-\ell < 0$, then the key tableau $K_{+}^{*} (\Lambda)$ does not exist for that $N$ and the corresponding dual $m$-symmetric Schur function is equal to 0.

 In order not to lose any generality, we will also define
\begin{equation} \label{dualSchursgen}
  \mathfrak s_{\Lambda}^*(x_1,\dots,x_N;\hat x_1,\dots, \hat x_m)
 = K_{(0^{N-\ell},\lambda_\ell,\dots,\lambda_1,a_m,\dots,a_1)}(x_{1},\dots,x_N,\hat x_1,\dots,\hat x_m) =
 \sum_{T \in \mathcal{S}^{*} (\Lambda) } x^T.  
\end{equation}

 We see from Lemma~\ref{KeyLambdaDual} that, in the definition \eqref{defS*} of $\mathcal S^*(\Lambda)$,
the condition $K_+(T)  \leq K_{+}^{*} (\Lambda)$ can never be violated
in any position of $K_{+}^{*} (\Lambda)$ occupied by a letter that does not belong to $\hat {\mathcal A}$ (since those cells are filled with the largest letters not in $\hat {\mathcal A}$).  It is thus sufficient to restrict ourselves to the positions occupied by the letters $\hat 1,\dots,\hat m$ in $K_{+}^{*} (\Lambda)$, which we write as
\begin{equation} \label{eqrestriction}
\mathcal S^*(\Lambda) =  \left \{ T \in {\rm Tab}_{\mathcal A}(\Lambda^{(0)})\,  : \,  K_+(T)  \big |_{\hat {\mathcal A} }   \leq  K_{+}^{*} (\Lambda) \big |_{\hat {\mathcal A} } \right \}. 
\end{equation}

We also need to introduce another key tableau associated to $\Lambda$ which, as we will see,  is essentially dual to 
$K_{+}^{*} (\Lambda)$.
\begin{definition}
\label{KeyLambda} The key tableau
$K_{+} (\Lambda)$ is the unique tableau such that, for $i=1,\dots,m$, every column to the left of the circle filled with an $i$ contains a letter $i$ and such that the remaining cells (supposing that there are $k$ of them in a given column) are filled with the letters $N,N-1,\dots,N-k+1$
 (if $N-k+1 \leq 0$, then the key tableau $K_{+}(\Lambda)$ does not exist for that $N$). 
\end{definition}

\begin{example}
As in the previous example, we use
        $$
 \Lambda  =   \tableau[scY]{  &  &  &   &  &   &   \bl \tcercle{3} \\ & &  & &  \bl \tcercle{1} \\ & & \\ & &  \bl \tcercle{4} \\  
 \\  \bl \tcercle{2}
}
$$
Since the circle labeled $4$ lies in column $3$, the first two columns contain a letter $4$. Likewise, the circle labeled $3$ lies in column $7$, so the first six columns contain a letter $3$, and since the circle labeled $1$ lies on the column $5$, the first four columns contain the letter $1$. Hence, after filling the remaining cells with the letters $N$ and $\bar N=N-1$, we get
$$
 K_+(\Lambda) =   \tableau[scY]{ 1 & 1 & 1 & 1  & 3 & 3   \\ 3 & 3 & 3  & 3 \\ 4 & 4 & N \\ \bar N & N  \\ N
}
$$
where ${\bar N } = N-1$.
\end{example}

\begin{remark} \label{remark*}
Comparing the constructions of $K_{+}(\Lambda)$ and $K_{+}^{*}(\Lambda)$, we see that if the entries in  column $k$ of $K_+(\Lambda)$ are 
$\{ l_1,\dots,l_r \}$, then those in  
column  $k$ of ${K_{+}^*(\Lambda)}$ are $\{l_1^*,\dots, l_r^*\}$,
where 
$$
l^*= 
\left\{
\begin{array}{ll}
\hat m +  \hat 1-\hat l & {\rm ~if~} l \in \{ 1,\dots,m \} \\
 l &  {\rm ~otherwise}
 \end{array}
 \right ..
$$
The map $*: \{1,\dots,N \}  \to \mathcal A, l \mapsto l^*$ will be called the $*$-operation.
\end{remark}
The following set will prove to be important:
\begin{equation}
\label{SetmSchur}
\mathcal{S} (\Lambda) =    \left \{ T  \in {\rm Tab}_{\{1,\dots,N \}}(\Lambda^{(0)}) \, : \, K_{+}  (T) \big |_m = K_{+} (\Lambda) \big |_m \right \},
\end{equation}
where $Q|_m$ refers to the skew tableau  obtained from $Q$ by only considering the letters $1,\dots,m$.  From this set, we define the functions
\begin{equation} \label{mschur2T}
{\mathfrak s}_{\Lambda} (x_1,\dots,x_N) = \sum_{T \in \mathcal{S} (\Lambda)} x^{T}.
\end{equation}
The relevance of the functions ${\mathfrak s}_{\Lambda} (x_1,\dots,x_N)$ will become clear later when we prove that they correspond to the case $t=0$ of the $m$-symmetric Schur function  ${s}_{\Lambda} (x_1,\dots,x_N;t)$.

\begin{remark}
\label{InclusionRemark}
The dual key polynomials defined in \eqref{defdualkey}
have the following combinatorial interpretation similar to that of the key polynomials given in \eqref{keysum}:
\begin{equation} \label{dualkeysum}
\hat K_\alpha(x_1,\dots,x_N) = \sum_{T} x^T,
\end{equation}
where the sum is over all tableaux $T \in {\rm Tab}_{\{1,\dots,N\}}$ of shape $\alpha^+$ such that $K_+(T) = K_+(\alpha)$.  It is thus immediate that, for
$\Lambda = ( \pmb a, \lambda)$, we have
$$
\mathfrak{s}_{\Lambda} (x_1,\dots,x_N) = \sum_{\beta}  \hat{{K}}_{\pmb a,\beta}(x_1,\dots,x_N),
$$
where $\beta$ ranges over all distinct permutations of $(\lambda_1,\dots,\lambda_\ell,0^{N-\ell})$. 
\end{remark}

\begin{remark} Suppose that $T \in {\rm Tab}_{\{1,\dots,N \}}$
has $a_i$ occurrences of the letter $i$, for $i=1,\dots,m$.  Observe
that there is a unique key tableau in ${\rm Tab}_{\{1,\dots,m \}}$ with such occurrences of the letter $i$, for $i=1,\dots,m$, and that this key tableau corresponds to
$K_{+} (\Lambda) \big |_m$. That is, to any
$T \in {\rm Tab}_{\{1,\dots,N \}}$ corresponds a unique $m$-partition $\Lambda$,
denoted $\Lambda(T)$,
such that $K_{+}  (T) \big |_m = K_{+} (\Lambda) \big |_m$.
The $m$-partition $\Lambda(T)$ can also be obtained explicitly in the following way.  Suppose that $T$ has shape $\mu$.
For $i=1,\dots,m$, let  $a_i$ be the number of $i$'s in the key tableau $K_+(T)$. We define $\Lambda(T)$ to be the $m$-partition whose diagram is obtained from $\mu$ by adding a circle filled with a $1$ in the uppermost row of size $a_1$, and then adding a circle filled with a $2$ in the uppermost row of size $a_2$ that does not already contain a circle, and so on until the $m$ circles have been added. It is indeed possible to do this process because $K_+(T)$, being a key tableau, has columns that are contained into each other.  Therefore, if
the letter $i$ appears $a_i$ times in $K_+(T)$, then they have to  occur in the first $a_i$ columns of $K_+(T)$, which implies that column $a_i+1$ of $K_+(T)$ is shorter than column $a_i$.
\end{remark}
From the previous remark, we immediately get that
$\mathcal{S} (\Lambda)$ can be rewritten in a simplified form as
\begin{equation}
\mathcal{S} (\Lambda) =    \left \{ T  \in {\rm Tab}_{\{1,\dots,N \}} \, : \, \Lambda(T)=\Lambda \right \}.
\end{equation}

\begin{example}
    Let $N=2$, $m=1$ and
$$
\Lambda = \tableau[scY]{&   \\   &
\bl \tcercle{1}  
}
$$
The tableaux contributing to $\mathfrak{s}_{\Lambda}^{*}(x)$ are then
\begin{equation}
\tableau[scY]{ 1 & 1  \\ 
2 & \bl
},
\qquad 
\tableau[scY]{ 1 & 2  \\ 
2 & \bl
},
\qquad 
\tableau[scY]{ 1 & 1  \\ 
\hat{1} & \bl
},
\qquad 
\tableau[scY]{ 1 & 2  \\ 
\hat{1} & \bl
}
\quad {\rm and} \quad
\tableau[scY]{ 2 & 2  \\ 
\hat{1} & \bl
}
\end{equation}
Therefore,
$$
\begin{array}{ll}
\mathfrak{s}_{\Lambda}^{*}(x_{1},x_{2};\hat{x}_{1}) = x_{1}^{2}x_{2} + x_{1}x_{2}^{2} + x_{1}^{2} \hat{x}_{1} + x_{1}x_{2}\hat{x}_{1} +x_{2}^{2}\hat{x}_{1} \quad \text{ and} \\
\\
\mathfrak{s}_{\Lambda}^{*}(x_{1},x_{2})= \mathfrak{s}_{\Lambda}^{*}(x_{1},x_{2};x_{1}) = 2x_{1}^{2}x_{2} + 2 x_{1}x_{2}^{2}+x_{1}^{3}
\end{array}.
$$
On the other hand, there is a single tableau contributing to $\mathfrak{s}_{\Lambda}(x)$
$$
\tableau[scY]{ 1 & 2  \\ 
2 & \bl
}
$$
which gives $\mathfrak{s}_{\Lambda}(x)=x_{1}x_{2}^{2}$. 
\end{example}

\subsection{Admissible pairs}
As we seek to prove that
\begin{equation} \label{niceeq}
  \frac{1}{\left[\prod_{i+j \leq m+1}(1- x_i \hat y_j) \right] \left[\prod_{i,j=1}^N (1- x_i y_j) \right] }  
  =  \sum_{\Lambda}  \mathfrak s_{\Lambda}  (x_1,\dots,x_N)  \mathfrak{s}_{\Lambda}^{*}(y_1,\dots,y_N;\hat y_1,\dots, \hat y_m)
\end{equation}
the elements in
\begin{equation} \label{eqGm}
 G_m = \bigcup_\Lambda \left(\mathcal{S} (\Lambda) \times {\mathcal{S}}^* (\Lambda) \right)  
\end{equation}  
will prove fundamental.  We will say that the pair of tableaux $(P,Q)$
is admissible if $(P,Q) \in G_m$.  We now fix the notation that will be in force for the rest of the section.

For a tableau $Q \in {\rm Tab}_{\{1,\dots,N \}}$, we let $K_+^*(Q) \in {\rm Tab}_{\mathcal A}$ be the tableau obtained by applying the $*$-operation described in Remark~\ref{remark*} on the entries in the columns of $K_+(Q)$ (and then reordering the entries so that the columns are that of a tableau).  

Recall that ${\mathcal C}_r(T)$ is the decreasing word corresponding to column $r$ of the key tableau $K_+(T)$.

For a tableau $P \in  {\rm Tab}_{\mathcal A}$, we let $\hat {\mathcal C}_r(P)$  be the subword
of ${\mathcal C}_r(P)$ obtained by keeping only the letters in $\hat {\mathcal A}$.

For a tableau  $Q \in {\rm Tab}_{\{1,\dots,N \}}$, we let  ${\mathcal C}_r^*(Q)$
be the word obtained by applying the $*$-operation on the letters of 
${\mathcal C}_r(Q)$ and then reordering the letters so that the resulting word  is decreasing.
We will then let $\hat {\mathcal C}_r^*(Q)$  be the subword
of ${\mathcal C}_r^*(Q)$ obtained by keeping only the letters in $\hat {\mathcal A}$.

We also recall that for $P \in {\rm Tab}_{\mathcal A}$,
$\hat W(P)$ stands for the set of subwords of $w(P)$
that only have letters in $\hat {\mathcal A}$.
\begin{example}
Let $m=4$ and
$$
T = \tableau[scY]{
1 & 1 & 2 \\
2 & 5 \\
3
}.
$$
In this case, we have
$$
K_{+}(T)= \tableau[scY]{
1 & 2 & 2 \\
2 & 5 \\
5
},
$$
 which implies that ${\mathcal C}_{1}(T)=521,\: {\mathcal C}_{2}(T)=52$ and ${\mathcal C}_{3}(T)=2$. This yields ${\mathcal C}_{1}^{*}(T)=\hat{4} \hat{3}5,\: {\mathcal C}_{2}^{*}(T)= \hat{3} 5 $ and ${\mathcal C}_{3}^{*}(T)= \hat{3}$, as well as  $\hat{\mathcal C}_{1}^{*}(T)=\hat{4} \hat{3},\: \hat{\mathcal C}_{2}^{*}(T)= \hat{3}  $ and $\hat{\mathcal C}_{3}^{*}(T)= \hat{3}$.
\end{example}
The following characterizations of admissibility will be used constantly in what follows.  The third one, which reads admissibility off suprema of decreasing subwords, is the form in which the condition can actually be propagated through a single RSK insertion.
\begin{proposition} \label{propcriteria}
  Let $P \in  {\rm Tab}_{\mathcal A}$ and $Q \in  {\rm Tab}_{\{1,\dots,N \}}$ be two tableaux of the same shape. Let also
  $l$ be the number of columns in $P$ (and $Q$).
The following statements are all equivalent to the admissibility of $(P,Q)$:
\begin{enumerate}
\item
$K_+(P) \big |_{\hat {\mathcal A} } \leq K_+^*(Q) \big |_{\hat {\mathcal A} }$
\item
  $
 \hat {\mathcal C}_r(P) \leq  \hat{\mathcal C}_r^*(Q) 
 $
 for all $r=1,\dots,l$.
\item
  $
 \sup \hat W(P_r) \leq  \hat{\mathcal C}_r^*(Q) 
 $
 for all $r=1,\dots,l$.
\end{enumerate}
\end{proposition}
\begin{proof}
  Suppose that $\Lambda(Q)=\Lambda$.  Since 
  $ K_+(Q)  \big |_m = K_{+} (\Lambda) \big |_m$,
  we then get  after applying the $*$-operation that
  $K_{+}^*(Q) \big |_{\hat {\mathcal A} }
   = K_{+}^* (\Lambda) 
   \big |_{\hat {\mathcal A} }$.  Hence, using \eqref{eqrestriction}, we get that $(P,Q)$ is admissible if and only if $K_+(P) \big |_{\hat {\mathcal A} } \leq K_+^*(Q) \big |_{\hat {\mathcal A} }$.

  Since $\hat {\mathcal C}_r(P)$ and  $\hat{\mathcal C}_r^*(Q)$ are essentially the columns of $K_+(P)$ and  $K_+^*(Q)$, respectively, restricted to their letters in $\hat {\mathcal A}$,  Statement {\it 2} is equivalent to Statement {\it 1}.

  By Proposition~\ref{IgualdadKeys}, we see that  $\sup  W(P_r)={\mathcal C}_r(P)$.  Given that the elements of  $\hat {\mathcal A}$ are the largest elements of $\mathcal A$ and that $\mathcal A \setminus \hat {\mathcal A}$ is finite,
 this yields  $\sup  \hat W(P_r)=\hat {\mathcal C}_r(P)$,
 and Statement {\it 3} is equivalent to Statement {\it 2}.
 \end{proof}

\subsection{Behavior of the recording tableau}
\label{BehaviourRecording}

Understanding the behavior of  $\hat {\mathcal C}_r^*(Q)$ will prove crucial in what follows. We begin by describing explicitly how appending a new letter to $Q$ affects $\mathcal C_r(Q)$.

\begin{remark}  \label{remarklength}
  Suppose that  $Q \in {\rm Tab}_{\{1,\dots,N \}}$ is a tableau whose largest entry is not larger than
  $m$. Then all the letters in $Q$ are smaller or equal to $m$, which implies that the length of $\hat{\mathcal{C}}_{r}^{*}(Q)$ is equal to the length of column $r$ of $Q$.
\end{remark}

\begin{proposition}
  \label{PropInsercion}
Suppose that  $Q \in {\rm Tab}_{\{1,\dots,N \}}$ is a tableau whose largest entry is not larger than
$i$.  Then let 
    $Q'$ be the tableau obtained from $Q$ by adding a letter  $i$ in column $\ell$ (we are supposing that it is possible to do so), where $\ell$
is such that     
there is no letter $i$ to the right of column $\ell$ 
in $Q$. We then have that
\begin{equation}
\label{InsercionRecording}
 {\mathcal C}_r(Q')
=
\left\{
 \begin{array}{ll}
   {\mathcal C}_r(Q) &\quad {\rm if~} r > \ell\\
   i{\mathcal C}_r(Q) &\quad {\rm if~} r =\ell \\ 
   {\mathcal C}_r'  & \quad  {\rm  if~ } r < \ell
\end{array}
\right.,
\end{equation}
where $ {\mathcal C}_r'$ is obtained from ${\mathcal C}_r(Q)$
by replacing by $i$ the largest entry of 
${\mathcal C}_r(Q)$ not in $ {\mathcal C}_\ell(Q)$ and then reordering the letters to get a decreasing word 
(observe that this amounts to doing nothing if this largest entry is $i$). 
\end{proposition}
\begin{proof}
  Throughout the proof we will use the fact that $\sup W(T_r)=
   {\mathcal C}_r(T)$ for any tableau $T$, which is the content of Proposition~\ref{IgualdadKeys}.  The three cases are treated in turn.

  When $r>\ell$, the tableaux $Q_r$ and $Q_r'$ coincide, so that $W (Q_r) =  W (Q_r')$ and $\sup W(Q_r) = \sup W(Q_r')$.

  When $r=\ell$, a word $w$ belongs to $W (Q_\ell')$ if and only if either  $w \in W (Q_\ell)$ or
  $w=i w'$ with $w'\in W(Q_\ell)$.  This immediately gives that
  $$
\sup W(Q_\ell') = i \sup W(Q_\ell) 
  $$
as claimed.
  
  Suppose finally that $r<\ell$.  If $W( Q_{r} )$ already contains a letter $i$, then $W(Q_r)=W(Q_r')$, so that once more $\sup W(Q_r) = \sup W(Q_r')$.  This is the degenerate situation described at the end of the statement, in which the largest entry of ${\mathcal C}_r(Q)$ not in ${\mathcal C}_\ell(Q)$ is already an $i$.

  We may thus assume that there is no $i$ in  $W( Q_r)$. Every word $w \in W( Q_r')$ is either $w \in W( Q_r)$ or $w=iw'$ with $w' \in W(Q_{\ell})$.
    Therefore
$$
\sup W( Q_r ' )  =  \sup \{ \sup W( Q_r) , i\sup W( Q_\ell) \}.
$$
Since $\sup W( Q_\ell) \subseteq \sup W( Q_r)$, the proposition follows immediately from Proposition~\ref{propaddone}.
\end{proof}

The previous proposition has the following consequence.
\begin{corollary}
\label{DesigualdadQ}
If $Q$ is a tableau that satisfies the conditions of Proposition~\ref{PropInsercion} then 
$$
 {\mathcal C}_r(Q)  \leq  {\mathcal C}_r(Q') 
\quad {\rm for~all~} r. 
$$
\end{corollary}
The behavior of $\hat {\mathcal C}_r^*(Q)$ under the addition of a letter is somewhat more complicated. 
\begin{corollary}
  \label{DesigualdadQstar}
 Let $Q$ be a tableau that satisfies the conditions of Proposition~\ref{PropInsercion}.
If  $i>m$ then
\begin{equation*}
\hat {\mathcal C}_r^*(Q')
=
\left\{
 \begin{array}{ll}
   \hat {\mathcal C}_r^*(Q) &\quad {\rm if~} r \geq \ell\\
       \hat {\mathcal C}_r' & \quad  {\rm  if~ } r  < \ell
\end{array}
\right. \quad,  
\end{equation*}
where $\hat {\mathcal C}_r'$ is obtained from $\hat {\mathcal C}_r^*(Q)$
by deleting the smallest entry of 
$\mathcal C_r^*(Q)$ not in $ \mathcal C_\ell^*(Q)$, if this entry belongs to $\hat{ \mathcal{A}}$
(and by doing nothing otherwise). 

If $i\leq m$, then
  \begin{equation*}
\hat {\mathcal C}_r^*(Q')
=
\left\{
 \begin{array}{ll}
   \hat {\mathcal C}_r^*(Q) &\quad {\rm if~} r > \ell\\ \hat
   {\mathcal C}_r^*(Q)(\hat m +\hat 1-\hat i) &\quad {\rm if~} r =\ell \\ 
    \hat {\mathcal C}_r' & \quad  {\rm  if~ } r  < \ell
\end{array}
\right. \quad,  
\end{equation*}
where $\hat {\mathcal C}_r'$ is obtained from $\hat {\mathcal C}_r^*(Q)$
by replacing by $\hat m +\hat 1-\hat i$ the smallest entry of 
$\hat {\mathcal C}_r^*(Q)$ not in $ \hat {\mathcal C}_\ell^*(Q)$ and then reordering the letters to get a decreasing word (observe that this amounts to doing nothing if this smallest entry is $\hat m +\hat 1-\hat i$). 
\end{corollary}
The analog of Corollary~\ref{DesigualdadQ} is then the following (the case
$r=\ell$ does not hold because the length increases when going from
$\hat {\mathcal C}_\ell^*(Q)$ to  $\hat {\mathcal C}_\ell^*(Q')$).  Observe that the inequality is reversed with respect to that of Corollary~\ref{DesigualdadQ}, the $*$-operation being order-reversing.
\begin{corollary}
\label{DesigualdadQ*}
If $Q$ is a tableau that satisfies the conditions of Proposition~\ref{PropInsercion} then 
$$
 \hat {\mathcal C}_r^*(Q')  \leq  \hat {\mathcal C}_r^*(Q) 
\quad {\rm for~all~} r\neq \ell .
$$
\end{corollary}

\begin{remark}
    Proposition~\ref{PropInsercion} shows that the right key tableau $K_{+}(Q')$ is completely determined by $K_{+}(Q)$, the inserted letter $i$, and the insertion column $\ell$; no further information about $Q$ is required.
\end{remark}

\begin{example}
Let $m=6$, $N \geq m$, and
$$
Q= 
\tableau[scY]{ 1   & 2  & 2  & 3 & 4  \\   2  & 3 & 4  \\   4 & 5 \\   5  
}
$$
Then
$$
K_{+}(Q) = \tableau[scY]{  2 & 3  & 3  & 4 & 4  \\  3  & 4 & 4  \\   4 & 5 \\   5   
} \qquad {\rm and}
\qquad
K_{+}^{*}(Q) = \tableau[scY]{  \hat{2} & \hat{2}  & \hat{3}  & \hat{3} & \hat{3}  \\  \hat{3}  & \hat{3} & \hat{4}  \\   \hat{4} & \hat{4} \\   \hat{5}   
}
$$
Since $Q$ contains no entry equal to $6$, we may append a $6$ to the end of the second row, obtaining
$$
Q ' = 
\tableau[scY]{ 1   & 2  & 2  & 3 & 4  \\   2  & 3 & 4 & 6  \\   4 & 5 \\   5  
}
$$
Then
$$
K_{+}(Q') =
\tableau[scY]{  2 & 3  & 4  & 4 & 4  \\  3  & 4 & 6 & 6  \\   4 & 6 \\   6   
}
\qquad {\rm and}
\qquad
K_{+}^{*}(Q') =
\tableau[scY]{  \hat{1} & \hat{1}  &  \hat{1} & \hat{1} & \hat{3}  \\  \hat{3}  & \hat{3} & \hat{3} & \hat{3}  \\   \hat{4} & \hat{4} \\   \hat{5}
}
$$
Now consider the same tableau $Q$, but with $m=3$. In this case,
$$
K_{+}(Q) = \tableau[scY]{  2 & 3  & 3  & 4 & 4  \\  3  & 4 & 4  \\   4 & 5 \\   5   
}
\qquad {\rm and}
\qquad
K_{+}^{*}(Q) = \tableau[scY]{  4 & 4  & 4  & 4 & 4  \\  5  & 5 & \hat{1}  \\  \hat{1} & \hat{1} \\   \hat{2}   
}
$$
Moreover,
$$
K_{+}(Q') =
\tableau[scY]{  2 & 3  & 4  & 4 & 4  \\  3  & 4 & 6 & 6  \\   4 & 6 \\   6   
}
\qquad {\rm and}
\qquad
K_{+}^{*}(Q') =
\tableau[scY]{  4 & 4 & 4 & 4 & 4  \\  6  & 6 & 6 & 6  \\   \hat{1} & \hat{1} \\   \hat{2}
}
$$
\end{example}

\subsection{The RSK algorithm and admissible pairs}

\label{Sec:RSK}

We now come to the technical core of the article: the compatibility between admissibility and a single step of the RSK insertion.  Recall from Section~\ref{ssectab} the row insertion algorithm $T \leftarrow j$ and the associated notion of insertion path.

We say that $(P',Q')=(P,Q) \leftarrow {i \choose j}$ is RSK-compatible if the following conditions are satisfied
\begin{itemize}
\item $(P,Q)$ and $(P',Q')$ are pairs of tableaux of the same shape in
  ${\rm Tab}_{\mathcal A} \times {\rm Tab}_{\{1,\dots,N \}}$.
\item $P'=P \leftarrow j$, the insertion of the letter $j$ in $P$.
\item  $Q$ is a subtableau of $Q'$, and the only cell in  $Q'/Q$ is filled with the letter $i$.
\item No letter in $Q$ is larger than $i$, and no letter $i$ in $Q$ occurs in a column to the
  right of the column in which $Q'/Q$ lies.
\end{itemize}  

We will also say that the biletter ${i \choose j}$ with $j \in \mathcal A$ and $i \in \{ 1,\dots,N\}$
is admissible if whenever $j \in \hat {\mathcal A}$ we have that $j \leq \hat m+\hat 1-\hat i$.

 The goal of this section is to show that
 \begin{equation} \label{equivalence}
\left[(P,Q)  \quad {\rm and} \quad {i \choose j}  \quad {\rm are~both~admissible} \right]  \iff (P',Q') \quad {\rm is~admissible},  
 \end{equation}
where we recall that admissible pairs $(P,Q)$ were defined after \eqref{eqGm}.
This will be a consequence of Propositions~\ref{PropInsPQ}, \ref{PropReciproca} and \ref{LimiteBuenoYMalo}.  Establishing them is where the work lies, as it relies on rather delicate properties of the words in $\hat W(P)$ and $\hat W(P')$.

We begin with a lemma providing the bounds on words on which all the subsequent arguments rest. 
\begin{lemma}
\label{LemIns}
    Consider $T'= T \leftarrow j$, the insertion of the letter $j$ in a tableau $T$. 
     Let $c_{k-1}, \ldots , c_{2}, c_{1}, c_{0} (=j)$ be the insertion path in $T'$ (which ends in row $k$ and column $\ell$). We have the following:
\begin{enumerate}     
\item[(1)]     
  For $r \leq \ell$,  let $w \in W(T_r')$ be such
  that $w$ has a letter in the first $k$ rows of $T'$.  
Then, $w=w_s \cdots w_1$ (with $s\geq k$) is such that 
$$
w_{k} w_{k-1} \ldots w_{2} w_{1} \leq c_{k-1} c_{k-2} \ldots c_{1} c_{0}.
$$
\item[(2)] If $w \in  W(T_r')$ is such that $w_i$ is weakly to the left of $c_s$ (in the same row) for some $c_s$ in the insertion path, 
and such that $w_i, w_{i-1}, \ldots, w_{\ell+1-i}$ are in consecutive rows (all within the first $k$ rows of $T$)
then
$$
w_i w_{i-1} \cdots w_{\ell+1-i} \leq c_s c_{s-1} \cdots c_{\ell+1-s}. 
$$
\end{enumerate}
\end{lemma}

\begin{proof}

We will only prove {\it (1)}, as {\it (2)} can be proven in a similar fashion.

The letter $c_{k-1}$ is the largest entry in row $k$ of $T'$ since it is the endpoint of the insertion algorithm. Hence, $w_{k} \leq c_{k-1}$.
By the insertion algorithm,  $c_{k-2}$ is the largest entry  smaller than $c_{k-1}$ in row $k-1$ of $T'$. This implies that 
$w_{k-1} \leq c_{k-2}$ since otherwise we would have the contradiction $w_{k-1} \geq c_{k-1} \geq w_k$ (recall that $w$ is a decreasing word since it belongs to $W(T_r')$). By the same reasoning, we conclude that  $w_{k-2} \leq c_{k-3}, \ldots , w_{2} \leq c_{1}, w_{1} \leq c_{0} (=j)$.
\end{proof}

The next two propositions allow us to construct, from a word 
$w \in W(P_r')$ a word $v \geq w$ that almost belongs to  $W(P_r)$ (it belongs to $W(P_r)$ if we do not consider its last letter).
\begin{proposition} \label{propvuj}
 Let  $c_{k-1}, \ldots , c_{2}, c_{1}, c_{0} (=j)$
  be the insertion path of  $P'=P \leftarrow j$. Given a word $w=w_t\cdots w_1  \in W(P_r')$, let $p$ be the largest integer such that $w_p$ intersects the insertion path (if there is no such integer, then $p=1$). Let $v=w_t\cdots w_p v_{p-1} \cdots v_1 \in W(P'_r)$, where $v_{p-1} \cdots v_1$ is defined in the following way: 
    if $w_i$ lies in row $s+1$ weakly to the left (in the same row) of $c_{s}$, then let $v_i=c_{s}$. Otherwise, let $v_i=w_i$.  We then have that  $w\leq v$,  with  either
  $v \in W(P_r)$ or $v=uj$ for some $u \in W(P_r)$.
  Observe that the proposition still holds with $W$ replaced everywhere by $\hat W$ since for any $w \in \hat W(P_r')$, the relation $w\leq v$  implies that $v$ also belongs to $\hat W(P'_r)$. 
\end{proposition}  
\begin{proof}
  We first show that $v$ is decreasing.  The only case which could potentially be problematic is when $v_{i+1}=w_{i+1}$ and $v_{i}=c_{s}$. Since $w_{i+1}$ is to the right of the insertion path, we have that $w_{i+1}$ is in the row (and to the right) of some $c_t$.  Hence, we have by the insertion algorithm that
  $v_i=c_{s} < c_t \leq w_{i+1}=v_{i+1}$.  It is then immediate by construction that $w \leq v$.  Since all the $c_s$'s belong to $P'_r$ we then conclude that
   $v \in W(P'_r)$.

  All the $v_i$'s (except possibly $v_1=c_0=j$) belong to $P_r$.  We will thus have that  $v \in W(P_r)$ or $v=uj$ with $u \in W(P_r)$ if we can show that they belong to distinct rows in $P_r$.  The only case to check is when
   $v_{i}=c_s$ and $v_{i-1}=w_{i-1}$ since in this case $c_s$ lies in $P_r$  one row above its position in $P'_r$.  But by Lemma~\ref{LemIns}, $w_{i-1}$ cannot be in the row that follows that of $w_i$, since otherwise we would have the contradiction that $w_{i-1} \leq c_{s-1} < c_s$ (that is, $w_{i-1}$ would not be to the right of the insertion path).
  \end{proof}  

\begin{example}
 If we insert the letter $j=3$ in  
 $$
 P=   \tableau[scY]{  1 & 2 & 4 & 5 \\  3  \\  6   
} 
 $$
then the insertion path (highlighted in blue) is given by
$$
P' = \tableau[scY]{  1 & 2 & {\color{blue} 3} & 5 \\  3 & {\color{blue} 4}  \\  6   
}
$$
Now, consider $w=42 \in W(P_{2}')$ as highlighted in red in $\tableau[scY]{  1 & {\color{red} 2} & 3 & 5 \\  3 & {\color{red} 4}  \\ 6
}$. The procedure in Proposition~\ref{propvuj} then gives us
 $v=43 \in W (P_{1}')$ as highlighted in red in $\tableau[scY]{  1 & 2 & {\color{red} 3} & 5 \\  3 & {\color{red} 4}  \\ 6
}$, where we see that $v=uj$ with $u=4 \in W(P_{2})$ and $j=3$.

If we had chosen instead  $w=632  \in W(P_{1}')$ as highlighted in red in 
$\tableau[scY]{  1 & {\color{red} 2} & 3 & 5 \\  {\color{red} 3} & 4  \\ {\color{red} 6}    
}$, which does not intersect the insertion path,
the procedure would then have given us $v = w \in W (P_{1}')$, with $v \in W ( P_{1} )$.
\end{example}

In the special case where $w$ has no entry below the last row of the insertion path, the construction simplifies.
\begin{proposition} \label{propvuj2}
 Let  $c_{k-1}, \ldots , c_{2}, c_{1}, c_{0} (=j)$
  be the insertion path of  $P'=P \leftarrow j$. Given a word $w=w_t\cdots w_1  \in W(P_r')$ that does not have any entry below row $k$, then  let $v=v_t\cdots v_1  \in W(P_r')$ be constructed in the following way:
    if $w_i$ lies in row $s+1$ weakly to the left (in the same row) of $c_s$, then let $v_i=c_s$. Otherwise, let $v_i=w_i$.  We then have that  $w\leq v$,  with  either
  $v \in W(P_r)$ or $v=uj$ for some $u \in W(P_r)$.
  Observe that the proposition still holds with $W$ replaced everywhere by $\hat W$, since for any $w \in \hat W(P_r')$, the relation
  $w\leq v$  implies that $v$ also belongs to $\hat W(P'_r)$. 
\end{proposition}

\begin{example}
 If we insert the letter $j=3$ in  
$$
 P=   \tableau[scY]{  1 & 2 & 4 & 5 \\  3 & 7  \\  6   
}
$$
then the insertion path (highlighted in blue) is given by
$$
P' = \tableau[scY]{  1 & 2 & {\color{blue} 3} & 5 \\  3 & {\color{blue} 4}  \\  6 & {\color{blue} 7}  
}
$$
Now, let $w=632  \in W(P_{1}')$ be the word highlighted in red in
$\tableau[scY]{  1 & {\color{red} 2} & 3 & 5 \\  {\color{red} 3} & 4  \\ {\color{red} 6} & 7  
}$. Then $v=uj=743$ is the word highlighted in red in $\tableau[scY]{  1 & 2 & {\color{red} 3} & 5 \\  3 & {\color{red} 4}  \\ 6 & {\color{red} 7}  
}$, where $u=74 \in W(P_{1})$. Observe that $w = 632 \leq 743 = v$.
If we had chosen instead $w=65 \in W(P_{1}')$ as highlighted in red in $\tableau[scY]{  1 & 2 & 3 & { \color{red}5 } \\  3 &  5  \\  {\color{red} 6}   & 7
}$,  then $v=75$ would  be the word highlighted in red in $\tableau[scY]{  1 & 2 & 3 & { \color{red}5 } \\  3 &  5  \\  6   & {\color{red} 7}
}$.
\end{example}
We are now in a position to show the forward implication ($\implies$) in
\eqref{equivalence}.

\begin{proposition}
  \label{PropInsPQ}
Let $(P',Q')=(P,Q) \leftarrow {i \choose j}$ be RSK-compatible.
If $(P,Q)$ and ${i \choose j}$ are both admissible, then
$(P',Q')$ is also an admissible pair.
\end{proposition}

\begin{proof}
  From Proposition~\ref{propcriteria}, we
  need to show that, for all $r$, any $w\in \hat W(P_r')$ is such that
$w \leq \hat {\mathcal C}_r^*(Q')$.
  By Proposition~\ref{propvuj}, we can find a $v \in \hat W(P_r')$ such that  $w \leq v$, where either  $v \in \hat W(P_r)$ or $v=uj$ with $u \in \hat W(P_r)$.  We assume throughout that $c_{k-1}, \ldots , c_{2}, c_{1}, c_{0} (=j)$
  is the insertion path of  $P'=P \leftarrow j$, and that $\ell$ is the column of the cell in $Q'/Q$.

Three observations will be used over and over, and it is worth isolating them once and for all.  First, the admissibility of $(P,Q)$ says that
  $\sup \hat W(P_r) \leq \hat {\mathcal C}_r^*(Q)$ (see Proposition~\ref{IgualdadKeys}), so that every $w' \in  \hat W(P_r)$
  satisfies $w' \leq \hat {\mathcal C}_r^*(Q)$.  Second, if $j \not \in \hat {\mathcal A}$ and $v=uj$, then $w \leq v$ forces $w\leq u$, all the letters of $w$ lying in $\hat {\mathcal A}$.  Third, $j \not \in \hat {\mathcal A}$ is automatic as soon as $i>m$: the admissibility of the biletter ${i \choose j}$ requires $j \leq \hat m+\hat 1-\hat i$, an inequality that no $j \in \hat{\mathcal A}$ can satisfy when $i>m$.

\smallskip
\noindent
{\it The case $i >m$ and $r \geq \ell$.} Corollary~\ref{DesigualdadQstar} gives here
  $\hat {\mathcal C}_r^*(Q')=\hat {\mathcal C}_r^*(Q)$.
  If $v \in \hat W(P_r)$, the first observation yields at once
  $w \leq v \leq \hat {\mathcal C}_r^*(Q)=\hat {\mathcal C}_r^*(Q')$.
If instead $v=uj$, the last two observations give $w\leq u$, and we conclude in the same way that $w \leq u \leq \hat {\mathcal C}_r^*(Q)=\hat {\mathcal C}_r^*(Q')$.

\noindent
The case $i>m$ and $r < \ell$ will be treated at the end of the proof.

\smallskip
\noindent
{\it The case $i \leq m$ and $r > \ell$.} Corollary~\ref{DesigualdadQstar} again gives  $\hat {\mathcal C}_r^*(Q')=\hat {\mathcal C}_r^*(Q)$, so that the situations $v\in \hat W(P_r)$, and $v=uj$ with $j \not \in \hat {\mathcal A}$, are settled exactly as above.  There remains the case $v=uj$ with $j \in \hat {\mathcal A}$.  Suppose then that $P_r$ has $s$ rows. We have that
$c_s c_{s-1} \cdots c_1 \in \hat W(P_r)$ since all the letters in the insertion path are larger than $j \in   \hat {\mathcal A}$.   We thus have that
$\hat {\mathcal C}_r^*(Q)=a_s a_{s-1} \cdots a_1$
for some word $a_s a_{s-1} \cdots a_1$
such that  $c_s c_{s-1} \cdots c_1\leq a_s a_{s-1} \cdots a_1$.
 Since $u$ has at most $s-1$ entries and $j =c_0 < c_1 \leq a_1$, we have that $u j \leq  a_s a_{s-1} \cdots a_1 $. Hence  $w \leq uj \leq  \hat {\mathcal C}_r^*(Q) = \hat {\mathcal C}_r^*(Q')$ as desired.

\smallskip
\noindent
{\it The case $i \leq m$ and $r=\ell$.} Corollary~\ref{DesigualdadQstar} now gives
 $\hat {\mathcal C}_\ell^*(Q')= \hat {\mathcal C}_\ell^*(Q) a_0$
with $a_0=\hat m + \hat 1-\hat i$, and in particular
$\hat {\mathcal C}_\ell^*(Q) \leq \hat {\mathcal C}_\ell^*(Q')$.
The chain $w\leq v \leq \hat {\mathcal C}_\ell^*(Q) \leq \hat {\mathcal C}_\ell^*(Q')$ thus disposes of the case $v\in \hat W(P_\ell)$, and the same chain with $u$ in place of $v$ disposes of the case $v=uj$ with $j \not \in \hat {\mathcal A}$.
We are left with $v=uj$ and $j\in \hat {\mathcal A}$,
 for which we immediately get that
$w \leq uj \leq \hat {\mathcal C}_r^*(Q)a_0$  since $j \leq a_0=\hat m + \hat 1-\hat i$ by the admissibility of ${i \choose j}$.

\smallskip
\noindent
{\it The case $i \leq m$ and $r < \ell$.} If column $r$ of $Q$ has length $t$,
 then we know by Remark~\ref{remarklength} (recall that we are in the case $i\leq m$) that $\hat {\mathcal C}_r^*(Q)$ is also of length $t$.
 Let $\hat {\mathcal C}_\ell^*(Q)= a_{k-1} \cdots a_1$
 and $\hat {\mathcal C}_r^*(Q)= b_t \cdots b_{s+1}b_s a_{s-1} \cdots a_1$ with $a_{s-1}<b_s < a_s$ ($\hat {\mathcal C}_\ell^*(Q)$ is a subword of $\hat {\mathcal C}_r^*(Q)$ since they are columns of a key tableau), so that
$b_s$ is the smallest letter in $\hat {\mathcal C}_r^*(Q)$ that is not in $\hat {\mathcal C}_\ell^*(Q)$.  
By Corollary~\ref{DesigualdadQstar}, we thus have in this case that
$$
\hat {\mathcal C}_r^*(Q')= b_t \cdots b_{s+1} a_{s-1} \cdots a_1 a_0,
$$
where $a_0=\hat m +\hat 1- \hat i$.
Suppose that $v \in \hat W(P_r)$. Given that $v \leq \hat {\mathcal C}_r^*(Q)$,
in order to show that
$v \leq \hat {\mathcal C}_r^*(Q')$  it suffices to check that if 
 $v=v_t \cdots v_s$, then $v_s \leq a_{s-1}$. Let $v'=v_k \cdots v_s$.
Given that $v_k$ lies in a row weakly above row $k$,  using the construction in Proposition~\ref{propvuj2}, we have that
$v_k \cdots v_s \leq x_k \cdots x_s$ with $x_k \cdots x_s \in \hat W(P_\ell)$.
This implies that $x_k \cdots x_s \leq \hat {\mathcal C}_\ell^*(Q)=a_{k-1} \cdots a_1$.
Hence $v_s \leq x_s \leq a_{s-1}$ as wanted, which gives that $w \leq v \leq \hat {\mathcal C}_r^*(Q')$.  If instead $v= uj$ with $u \in \hat W(P_r)$, we have similarly that $u \leq  \hat {\mathcal C}_r^*(Q')$.  Given that $j \leq  \hat m+\hat 1-\hat i=a_0$ by the statement of the theorem, we have that $w\leq uj \leq  \hat {\mathcal C}_r^*(Q')$. 

\smallskip
\noindent
{\it The case $i>m$ and $r < \ell$.} It remains to treat the case that we had set aside.  We have by Corollary~\ref{DesigualdadQstar}, that $\hat {\mathcal C}_r^*(Q') = \hat {\mathcal C}_r'$, where $\hat {\mathcal C}_r'$ is obtained from $\hat {\mathcal C}_r^*(Q)$ by deleting the smallest entry of 
$ {\mathcal C}_r^*(Q)$ not in $ {\mathcal C}_\ell^*(Q)$ (if this smallest entry belongs to $\hat{\mathcal{A}}$). If this entry is not in $\hat{\mathcal{A}}$, then $\hat {\mathcal C}_r' =\hat {\mathcal {C}}_r^*(Q)$ and we can proceed as in the case $i>m$ and $r \geq \ell$.  We thus only have left to prove the case when the smallest entry is in $\hat{\mathcal{A}}$.  As we will see, this case is very similar to the case $i\leq m$ and $r <\ell$.  Since the smallest entry of 
${\mathcal C}_r^*(Q)$ not in $ {\mathcal C}_\ell^*(Q)$ belongs to $\hat{\mathcal{A}}$, 
the entries of 
${\mathcal C}_r^*(Q)$ and $ {\mathcal C}_\ell^*(Q)$
in $\mathcal A \setminus \hat{\mathcal{A}}$ are exactly the same
(since those entries are the smallest entries).
Suppose that there are $p-1$ such entries in $\mathcal A \setminus \hat{\mathcal{A}}$.
We thus have
 $\hat {\mathcal C}_\ell^*(Q)= a_{k-1} \cdots a_p$
 and $\hat {\mathcal C}_r^*(Q)= b_t \cdots b_{s+1}b_s a_{s-1} \cdots a_p$ with $a_{s-1}<b_s < a_s$ ($\hat {\mathcal C}_\ell^*(Q)$ is a subword of $\hat {\mathcal C}_r^*(Q)$ since they are columns of a key tableau), so that
$b_s$ is the smallest letter in $\hat {\mathcal C}_r^*(Q)$ that is not in $\hat {\mathcal C}_\ell^*(Q)$.  
By Corollary~\ref{DesigualdadQstar}, we thus have in this case that
$$
\hat {\mathcal C}_r^*(Q')= b_t \cdots b_{s+1} a_{s-1} \cdots a_p.
$$
Suppose that $v \in \hat W(P_r)$. By the same argument as in the case $i\leq m$ and $r <\ell$, we can deduce that $w \leq v \leq \hat {\mathcal C}_r^*(Q')$.  If instead $v= uj$ with $u \in \hat W(P_r)$, we have similarly that $u \leq  \hat {\mathcal C}_r^*(Q')$.  Given that 
$j \not \in \hat {\mathcal A}$ (as in the case $i>m$ and $r \geq \ell$ above), it is immediate that
$w\leq u \leq  \hat {\mathcal C}_r^*(Q')$
since $w\leq v=uj$ and 
$u \in \hat W(P_r)$. 
\end{proof}

Before being able to show the backward implication ($\Longleftarrow$) in
\eqref{equivalence}, we need a few elementary results.  The first is an easy consequence of Lemma~\ref{SupKnuthEquivalent}.
\begin{lemma}
\label{LemaDesfasaje} Given
a tableau  $P$ and a letter $j$, let  $P'=P \leftarrow j$ be such that the insertion path ends in column $\ell$.  For all $r \leq \ell$ we have that
$\sup W(P_r) \leq \sup W(P_r')$ and that $\sup \hat W(P_r) \leq \sup \hat W(P_r')$.
\end{lemma}
\begin{proof}
If $r \leq \ell$, then $w(P_{r})j$ is Knuth equivalent to $w(P_{r}')$. Since $W(P_{r}) \subset W(P_{r}) \cup  W(P_{r})j$, we then have from Lemma~\ref{SupKnuthEquivalent} that
$$
\sup W(P_r)  \leq \sup \bigl(  W(P_{r}) \cup  W(P_{r})j \bigr)= \sup W(P_r').
$$
The proof is the same when we replace $W$ by $\hat W$.
\end{proof}
The next lemma will allow us to assume that the length of the words in $w \in \hat W(P_r)$ is not larger than the length of $ \hat {\mathcal C}_r^*(Q)$ when we insert a letter $i \leq m$ in $Q$.

\begin{lemma}
  \label{PropReciproca0}
  Let $(P',Q')=   (P,Q) \leftarrow {i \choose j}$ be RSK-compatible, with $i \leq m$.
  Suppose that, for a given $r$, there exists a word $w \in \hat W(P_r)$
 whose length is larger than the length of $ \hat {\mathcal C}_r^*(Q)$.
    Then the pair $(P',Q')$ is not admissible.
\end{lemma}
\begin{proof}
  By Remark~\ref{remarklength}, if $i \leq m$,
  we have that the length of $\hat {\mathcal C}_r^*(Q)$ is equal to the length of column $r$ of $Q$. Since $P$ and $Q$ have the same shape, it is thus impossible to find a word $w \in \hat W(P_r)$ whose length is larger than the length of 
$ \hat {\mathcal C}_r^*(Q)$.
\end{proof}  

The next two propositions will prove the backward implication 
($\Longleftarrow$) in  \eqref{equivalence}.
\begin{proposition}
  \label{PropReciproca}
 Let $(P',Q')=   (P,Q) \leftarrow {i \choose j}$ be RSK-compatible.
If the pair $(P,Q)$ is not admissible, then neither is the pair $(P',Q')$.
\end{proposition}

\begin{proof}

For a decreasing word $w=w_1 \cdots w_\ell$, it  will be convenient to use the notation  $(w)_s=w_s$.  Throughout the proof, $\ell$ denotes the column of the cell in $Q'/Q$, and we suppose that $\sup \hat W(P_r) \not\leq \hat {\mathcal C}_r^*(Q)$ for some $r$.  The three possible positions of $r$ relative to $\ell$ are treated separately.

\smallskip
\noindent
{\it The case $r < \ell$.}  Lemma~\ref{LemaDesfasaje} gives 
   $
    \sup \hat W(P_r)  \leq  \sup \hat W(P_{r}') 
    $.
    This implies that $\sup \hat W (P_r') \not\leq \hat {\mathcal C}_r^*(Q') $ since otherwise we would get the contradiction from Corollary~\ref{DesigualdadQstar}
that
    $$
\sup \hat W(P_r) \leq \sup \hat W(P_r' ) \leq \hat {\mathcal C}_r^*(Q') 
\leq \hat {\mathcal C}_r^*(Q).
$$

\smallskip
\noindent
{\it The case $r=\ell$.} We have again that $w(P_{\ell}')$ is Knuth equivalent to $w(P_{\ell})j$, which implies that $\sup \hat W(P_\ell)  \leq \sup \hat W(P_{\ell}')$  by Lemma~\ref{LemaDesfasaje}.
If we suppose that $\sup \hat W (P_\ell') \leq  \hat {\mathcal C}_r^*(Q')$,
then  we get from Corollary~\ref{DesigualdadQstar} that
    $$
\sup \hat W(P_\ell) \leq \sup \hat W(P_\ell' ) \leq  \hat {\mathcal C}_\ell^*(Q') 
 =  \hat {\mathcal C}_\ell^*(Q) (\hat m+\hat 1-\hat i) \qquad {\rm if~} i \leq m
$$
or
 $$
\sup \hat W(P_\ell) \leq \sup \hat W(P_\ell' ) \leq  \hat {\mathcal C}_\ell^*(Q')  =  \hat {\mathcal C}_\ell^*(Q)  \qquad {\rm if~} i > m.
$$

In the first case, using the previous lemma, we can assume that there is no word $w \in \hat W(P_r)$ whose length is larger than the length of $ \hat {\mathcal C}_r^*(Q)$ since otherwise  $(P',Q')$ would not be admissible. Hence we can assume that, for all $r$,  the length of $\sup \hat W(P_r )$  is not larger than the length of $ \hat {\mathcal C}_r^*(Q)$.
In both cases, it leads to the contradiction that
$\sup \hat W(P_\ell) \leq  \hat {\mathcal C}_\ell^*(Q) $ from which we conclude that $\sup \hat W (P_\ell') \not\leq  \hat {\mathcal C}_\ell^*(Q')$.

\smallskip
\noindent
{\it The case $r > \ell$.}
Since $\sup \hat W(P_r) \not\leq \hat {\mathcal C}_r^*(Q) $, we can suppose that
$( \sup \hat W(P_r))_{s} > (  \hat {\mathcal C}_r^*(Q) )_{s}$ for some $s$. Let  $w=w_{s} \ldots w_{1} \in \hat W(P_r)$ be such that  $(w)_s=w_{1} = ( \sup \hat W(P_r))_{s} $.
  Suppose first that $w$ does not intersect the insertion path. Then $w \in \hat W(P_r')$, so that
$$
  ( \sup \hat W(P_r') )_{s} \geq (w)_{s} = ( \sup \hat W(P_r) )_{s} >
( \hat {\mathcal C}_r^*(Q) )_{s} \geq  ( \hat {\mathcal C}_r^*(Q') )_{s}
$$
from Corollary~\ref{DesigualdadQstar}. We thus have that
$\sup \hat W (P_r') \not\leq  \hat {\mathcal C}_r^*(Q')$ in that case.

Suppose now, on the contrary, that $w$ intersects the insertion path $c_{k-1} \ldots c_{1}$ in $P$. Let $t$ be the smallest integer such that
    $w_{t}$  intersects the insertion path, and suppose that $w_t=c_{p-1}$. Observe that the entry $w_t$ lies in row  $p-1$  in  $P$ while it lies in row $p$ in $P'$.
    
We consider the word $u=c_{k-1} \cdots c_p w_t \cdots w_1$ of length  $(k-p)+t$.   The word $u$ is decreasing since
$c_{k-1} \cdots c_1$ is decreasing, $w$ is decreasing and $c_p > c_{p-1} =w_t$. We also have that  $u \in \hat W(P_\ell')$ since $w_t$ lies in the row above that of $c_p$ in $P'$ and since 
$w_{t-1},\dots, w_1$ belong to $P'$ given that they do not intersect the insertion path.

Suppose that $P_r$ has $q$ rows.
Within rows  $q$ and $p+1$ of $P'$  lie the decreasing words  $c_{q-1} \ldots c_p$ and $w_{s} \dots w_{t+1}$. Since the former has a letter in each of those rows, we have that
 $ s-t \leq q-p $, or equivalently, that $k-q+s \leq k-p+t$. Since $u$ is a decreasing word, we thus have that
$ (u)_{k-q+s} \geq (u)_{k-p+t} =w_{1}=(w)_s$.  Because $u \in \hat W(P_\ell')$,  this means that   $(\sup \hat W(P_\ell') )_{k-q+s} \geq  (u)_{k-q+s} \geq (w)_s$.  Hence, Proposition~\ref{PropInsercion} yields that
$$
(\sup \hat W(P_\ell') )_{k-q+s}   \geq (w)_{s} >
 (\hat {\mathcal C}_r^*(Q))_s   \geq  (\hat {\mathcal C}_r^*(Q'))_s \geq  (\hat {\mathcal C}_\ell^*(Q'))_{k-q+s},
$$
which implies that
$\sup \hat W(P_\ell') \not \leq  \hat {\mathcal C}_\ell^*(Q')$.
Note that $ (\hat {\mathcal C}_r^*(Q'))_s \geq  (\hat {\mathcal C}_\ell^*(Q'))_{k-q+s}$
since $\hat {\mathcal C}_r^*(Q')$ is a subword of $\hat {\mathcal C}_\ell^*(Q')$
and since the difference in length between the two words is 
at most $k-q$.
\end{proof}

\begin{example}
    Let $m=3$,
 $$
\left(P , Q \right)=  \left(  \tableau[scY]{   \hat{2} & \hat{3}  \\  \hat{3}    
} , 
 \tableau[scY]{  1 & 1  \\  3     
}
\right) \text{ and } (i,j)=(4,1).
$$
Then $\mathcal{C}^{*}_{1}(Q)= \hat{3}\hat{1}$ so $\hat{3}\hat{2} \in W (P_{1})$ breaks the condition.
$$
(P',Q') = \left( \tableau[scY]{   {\color{blue}  1 } & \hat{3}  \\   {\color{blue} \hat{2}}  \\   {\color{blue} \hat{3}}  
} ,  \tableau[scY]{  1 & 1  \\  3   \\ 4
}
\right)
$$
Observe that $(4,1)$ does not create any problem, being an admissible biletter. We can pick the word $w$ highlighted in red $ \tableau[scY]{    1  & \hat{3}  \\   {\color{red} \hat{2}}  \\   {\color{red} \hat{3}}
}$. Observe $\mathcal{C}^{*}_{1}(Q')= \hat{3}\hat{1}$ so $\hat{3}\hat{2} \in W (P_{1}')$ breaks the condition.

\end{example}

\begin{example}
    Let $m=3$,
 $$
\left(P , Q \right)=  \left(  \tableau[scY]{   \hat{2} & \hat{3}  \\  \hat{3}    
} , 
 \tableau[scY]{  1 & 1  \\  4     
}
\right) \text{ and } (i,j)=(5,1).
$$
Then $\mathcal{C}^{*}_{1}(Q)=\hat{3}$ so $\hat{3}\hat{2} \in W (P_{1})$ breaks the condition because it is larger than it.
$$
(P',Q') = \left( \tableau[scY]{   {\color{blue}  1 } & \hat{3}  \\   {\color{blue} \hat{2}}  \\   {\color{blue} \hat{3}}  
} ,  \tableau[scY]{  1 & 1  \\  4   \\ 5
}
\right)
$$
Observe that $(5,1)$ does not create any problem, being an admissible biletter. We can pick the word $w$ highlighted in red $ \tableau[scY]{    1  & \hat{3}  \\   {\color{red} \hat{2}}  \\   {\color{red} \hat{3}}
}$. Observe that $\mathcal{C}^{*}_{1}(Q')= \hat{3}$ so $\hat{3}\hat{2} \in W (P_{1}')$ breaks the condition.

\end{example}

\begin{proposition}
  \label{LimiteBuenoYMalo}
 Let $(P',Q')=   (P,Q) \leftarrow {i \choose j}$ be RSK-compatible. 
 If  the biletter ${i \choose j}$ is not admissible then neither is the pair $(P',Q')$.
\end{proposition}
\begin{proof}
First observe that  $j \in \hat {\mathcal A}$ since ${i \choose j}$ is not admissible.
Now, let $c_{k-1},\dots,c_1,c_0(=j)$ be the insertion path.
Given that $j \in \hat {\mathcal A}$,
$c_{k-1} \dots c_{1}j$ is a word of length $k$ in $\hat W(P_\ell')$, where
$\ell$ is the column of the cell in $Q'/Q$.
If $i \leq m$, we then have that $j > \hat m+\hat 1-\hat i$ since ${i \choose j}$ is not admissible. Hence,   $(P',Q')$ is not admissible since
$(\sup \hat W(P_\ell') )_{k} \geq j > \hat m+\hat 1-\hat i = (\hat {\mathcal C}_\ell^*(Q'))_k$ by Corollary~\ref{DesigualdadQstar}.
Finally, if $i>m$, we have  that the length of $\hat {\mathcal C}_\ell^*(Q')$ is smaller than $k$ by
Corollary~\ref{DesigualdadQstar} (which says that the length of  $\hat {\mathcal C}_\ell^*(Q')$ is that of  $\hat {\mathcal C}_\ell^*(Q)$).  Hence,
$(P',Q')$ cannot be admissible
given that
$\sup \hat W(P_\ell')$ has length $k$. 
\end{proof}

\subsection{The RSK correspondence and its consequences}

Having established \eqref{equivalence} one insertion at a time, we can now iterate and reap the consequences.

An  admissible biword is a biword of the form
$
\left(
\begin{array}{cccc}
i_1 & i_2 & \cdots & i_r \\
j_1 & j_2 & \cdots & j_r
\end{array}  
\right)
$
where every biletter ${i_k \choose j_k}$ is admissible.  The order of the biletters within a biword is immaterial; the biword is said to be in lexicographic order if
$i_1 \leq i_2 \leq \cdots \leq i_r$ and $j_k\leq j_{k+1}$ whenever $i_k = i_{k+1}$.

The following families of biwords will be useful.
\begin{itemize}
\item  $B_{{\mathcal A}}$ is the set of biwords in the admissible biletters  ${i \choose j}$
\item  $B_{\hat {\mathcal A}}$ is the set of biwords in the admissible biletters  ${i \choose j}$  such that $j \in \hat {\mathcal A}$
\item  $B_{\{1,\dots,N \}}$ is the set of biwords in the biletters ${i \choose j}$ such  that $i,j\in \{1,\dots,N \}$
\item  $ {\bar B_{{\mathcal A}}}$ is the set of biwords in the biletters ${i \choose j}$ such  that $i\in \{1,\dots,N \}$, and $j \in \mathcal A$ (without any admissibility condition).
\end{itemize}

The RSK correspondence provides a bijection between the set of biwords in $\bar B_{\mathcal A}$
and pairs of tableaux of the same shape in ${\rm Tab}_{\mathcal A} \times {\rm Tab}_{\{1,\dots,N\} }$.  From the biword $
\left(
\begin{array}{cccc}
i_1 & i_2 & \cdots & i_r \\
j_1 & j_2 & \cdots & j_r
\end{array}  
\right)
$ in lexicographic order, one builds successively a pair $(P,Q)=(P^{(r)},Q^{(r)})$
from the empty pair by repeatedly applying the RSK insertion
$(P^{(k)},Q^{(k)})=(P^{(k-1)},Q^{(k-1)}) \leftarrow {i_k \choose j_k}$.  The lexicographic order guarantees that each of these steps is RSK-compatible.

Recalling that $G_m$ stands for the set of admissible pairs of tableaux, our previous results can now be assembled into the following statement.
\begin{theorem}
\label{TeoremaKeys}
The RSK correspondence, when restricted to $B_{{\mathcal A}}$,  provides a bijection between the sets $B_{{\mathcal A}}$ and $G_{m}$.
\end{theorem}
\begin{proof}
  Let $(P',Q')=   (P,Q) \leftarrow {i \choose j}$ be RSK-compatible.
 As mentioned earlier,  Propositions~\ref{PropInsPQ}, \ref{PropReciproca} and \ref{LimiteBuenoYMalo} tell us that
  $$
\left[(P,Q)  \quad {\rm and} \quad {i \choose j}  \quad {\rm are~both~admissible} \right]  \iff (P',Q') \quad {\rm is~admissible}.  
  $$
Starting from a biword  $\mathbf b \in B_{{\mathcal A}}$ in lexicographic order and applying the RSK insertion one biletter at a time, the forward implication guarantees that the resulting pair $(P,Q)$ lies in $G_m$.  Conversely, undoing the insertions one at a time, the backward implication guarantees that the biword recovered from a pair $(P,Q) \in G_m$ is admissible, that is, that it belongs to $B_{{\mathcal A}}$.
\end{proof}

It is well known that
$$
\frac{1}{\prod_{i,j=1}^N (1- x_i y_j) } = \sum_{\mathbf b \in B_{\{1,\dots,N \}} }  (xy)^{\mathbf b},
$$
where $(xy)^{\mathbf b}$ is the monomial in the variables $x_1,\dots,x_N,y_1,\dots,y_N$
whose power of $x_i$ (resp. $y_i$) is the number of occurrences of the letter $i$ in the upper (resp. lower) row of the biword
$\mathbf b$.  Similarly, it can be easily seen that
$$
\frac{1}{\prod_{i+j\leq m+1} (1- x_i \hat y_j) } = \sum_{\mathbf b \in B_{\hat {\mathcal A}} }  (x\hat y)^{\mathbf b}.
$$
Combining those two expansions, we get that
\begin{equation} \label{xyy}
 \frac{1}{\left[\prod_{i+j \leq m+1}(1- x_i \hat y_j) \right] \left[\prod_{i,j=1}^N (1- x_i y_j) \right] }  
  =  \sum_{\mathbf b \in B_{ {\mathcal A}} }  (xy\hat y)^{\mathbf b},
\end{equation}
where  $(xy\hat y)^{\mathbf b}$ is the monomial in the variables $x_1,\dots,x_N,y_1,\dots,y_N,\hat y_1,\dots,\hat y_m$
whose power of $x_i$ (resp. $y_i$) is the number of occurrences of the letter $i$, for $i \in \{ 1,\dots,N\}$,  in the upper (resp. lower) row of the biword $\mathbf b$, and whose power of  $\hat y_{i}$  is the number of occurrences of the letter $\hat i$, for $\hat i \in \hat {\mathcal A}$,  in the lower row of the biword $\mathbf b$.

Owing to \eqref{xyy}, Theorem~\ref{TeoremaKeys}  has this important corollary.
\begin{corollary} \label{coromain} The following generalization of the Cauchy identity holds
\begin{equation} \label{cauchygene}
  \frac{1}{\left[\prod_{i+j \leq m+1}(1- x_i \hat y_j) \right] \left[\prod_{i,j=1}^N (1- x_i y_j) \right] }  
  =  \sum_{\Lambda}  \mathfrak s_{\Lambda}  (x_1,\dots,x_N)  \mathfrak{s}_{\Lambda}^{*}(y_1,\dots,y_N;\hat y_1,\dots, \hat y_m).
\end{equation}
\end{corollary}
Letting $y_i=0$ for $i=1,\dots,N$ amounts to restricting the correspondence to the hatted alphabet.  What survives is precisely Lascoux's nonsymmetric Cauchy identity \eqref{nonsymCauchy} for Demazure characters and atoms \cite{Lascoux2003}, of which we thus obtain a proof relying on nothing but ordinary RSK insertion.
\begin{corollary} \label{coroKeys} We have that
  $$
  \frac{1}{\prod_{i+j \leq m+1}(1- x_i \hat y_j)}   =\sum_{\pmb a}
 \hat K_{\pmb a} (x_1,\dots,x_m)  K_{ {\mathbf \omega}_m (\pmb a)} (\hat y_1,\dots,\hat y_m).   
$$
\end{corollary}
\begin{proof}
  Letting $y_i=0$ for $i=1,\dots,N$ means that in $\mathcal S^*(\Lambda)$ we will only admit tableaux $T \in {\rm Tab}_{\hat {\mathcal A}}$. From \eqref{eqrestriction}, we get that $K_+(T) \leq K_+^*(\Lambda)\big |_{\{\hat 1,\dots,\hat m \}}$
  is only possible if $K_+^*(\Lambda)$ only has entries in $\hat {\mathcal A}$, which implies that $\Lambda$ can only be of the form $\Lambda=(\pmb a; \emptyset)$. In this case, we get from \eqref{dualSchursgen} that
  $$
  \mathfrak{s}_{(\pmb a;\emptyset)}^{*}(\hat y_1,\dots, \hat y_m)= K_{(0^N,a_m,\dots,a_1)}(0,\dots,0;\hat y_1,\dots,\hat y_m) =
  K_{ {\mathbf \omega}_m (\pmb a)} (\hat y_1,\dots,\hat y_m).   
 $$
 On the other hand, we get that
 $$
 \mathfrak{s}_{(\pmb a;\emptyset)}(x_1,\dots, x_m)= \sum_T x^T, 
 $$
 where the sum is over all tableaux $T \in {\rm Tab}_{\{1,\dots,m \}}$ such that
 $
 K_+(T)=K_+(\pmb a)
 $ (given that none of the letters in $K_+(\pmb a)$ are larger than $m$). Comparing with \eqref{dualkeysum}, we recognize a dual key polynomial:
 $$
 \mathfrak{s}_{(\pmb a;\emptyset)}(x_1,\dots, x_m)= \hat K_{\pmb a}(x_1,\dots,x_m). 
$$
\end{proof}

Letting $\hat y_1=y_1,\dots, \hat y_m=y_m$  in  \eqref{cauchygene}, and 
using the definition
\eqref{dualSchur} of the dual $m$-symmetric Schur function at $t=0$,
we obtain a further Cauchy identity, this time in the original variables.
\begin{corollary} \label{coroDual} The following generalization of the Cauchy identity holds
\begin{equation}
  \frac{1}{\left[\prod_{i+j \leq m+1}(1- x_i  y_j) \right] \left[\prod_{i,j=1}^N (1- x_i y_j) \right] }  
  =  \sum_{\Lambda}  \mathfrak s_{\Lambda}  (x_1,\dots,x_N)  \mathfrak{s}_{\Lambda}^{*}(y_1,\dots,y_N).
\end{equation}
\end{corollary}  
Finally, comparing the previous equation with \eqref{genCauchy} in the case of $N$ variables, we get that  $\mathfrak s_{\Lambda}  (x_1,\dots,x_N)$ is an $m$-symmetric Schur function at $t=0$.
\begin{corollary} For every $m$-partition $\Lambda$, we have that
  $$
\mathfrak s_{\Lambda}  (x_1,\dots,x_N) = s_{\Lambda}  (x_1,\dots,x_N;0).
  $$ 
\end{corollary}

\section[Almost Symmetric Schur functions]{Connection with the almost symmetric Schur functions}
\label{Sec:AlmostSym}

In \cite{AlmostSym}, the \emph{almost symmetric Schur function} $\mathfrak{as}_{(\pmb a|\lambda)}(x_1,x_2,\dots)$ was introduced. It is indexed by a weak composition  $\pmb a=(a_1,\dots,a_m)$ of any length $m$ such that $a_m\neq 0$ and a  partition $\lambda$. The almost symmetric Schur functions form a basis of the space $\mathcal A \mathcal S$ of almost symmetric functions which, in our language, corresponds to
$$\mathcal A \mathcal S=
\bigcup_{m \geq 0} R_m, 
$$
where $R_m$ is the ring of $m$-symmetric functions.

Since we want to consider the almost symmetric 
functions $\mathfrak{as}_{(\pmb a|\lambda)}(x)$
as a basis of $R_m$, we will extend their definition to include the case where $a_m=0$. This family of almost symmetric Schur functions will be naturally indexed by   
 $m$-partitions and denoted
$\mathfrak{as}_\Lambda (x)$.
We will see in this section how they relate to the 
 $m$-symmetric Schur basis $\{\mathfrak{s}_{\Lambda}(x)\}_{\Lambda}$.

We will need the symmetrization operator 
\begin{equation} \label{eqW}
W_{m-1}^{(N)}= (\pi_{N-1} \cdots \pi_m)(\pi_{N-1} \cdots \pi_{m+1}) \cdots (\pi_{N-1} \pi_{N-2}) \pi_{N-1}, 
\end{equation}
which is simply a shifted version of the usual operator $\pi_{w_0}$, where $w_0$ is the longest permutation in $S_N$. By idempotency of the $\pi_i$'s, we have that
\begin{equation} \label{eqpiw0}
W_{m-1}^{(N)} \pi_i = \pi_i W_{m-1}^{(N)} = W_{m-1}^{(N)}
\end{equation}
for any $i=m,\dots,N-1$.

We now define the almost symmetric Schur functions in $N$ variables (the infinite case is then obtained via an inverse limit). Recall that the key polynomials were defined in \eqref{defkey}.
\begin{definition} \label{defirecursion}
For a weak composition $\pmb a=(a_1,\dots,a_m)$ and a partition $\lambda=(\lambda_1,\dots,\lambda_\ell)$, the
almost symmetric Schur function $\mathfrak{as}_{(\pmb a;\lambda)}(x_1,\dots,x_N)$ is defined recursively as follows.
\begin{enumerate}
    \item If $\lambda=\emptyset$, then
    \[
    \mathfrak{as}_{(\pmb a;\emptyset)}(x_1,\dots,x_N)
    = K_{\pmb a}(x_1,\dots,x_m).
    \]
    \item If $a_m\ge \lambda_1$ then
    \[
     \mathfrak{as}_{(a_1,\dots,a_{m-1} ; \mu)}(x)
    =
    W_{m-1}^{(N)}\,
    \mathfrak{as}_{(\pmb a;\lambda)}(x).
    \]
    where $\mu=(a_m, \lambda_1,\dots,\lambda_\ell)$.
\end{enumerate}
\end{definition}
Remarkably, the $m$-symmetric Schur functions $\mathfrak s_\Lambda(x)$ obey the very same recursion, with the key polynomials replaced by the dual key polynomials.  This is what will allow us to compare the two families. 

\begin{proposition} \label{proporecursion}
The family $\{\mathfrak{s}_{\Lambda}(x)\}_{\Lambda}$ is characterized by the following two properties:
\begin{enumerate}
    \item If $\lambda=\emptyset$, then
    \[
    \mathfrak{s}_{(\pmb a;\emptyset)}(x_1,\dots,x_N)
    =
   \hat K_{\pmb a}(x_1,\dots,x_m).
   \]

    \item If $a_m\ge \lambda_1$ then
    \[
    \mathfrak{s}_{(a_1,\dots,a_{m-1} ; \mu)}(x)
    =
    W_{m-1}^{(N)}\,
    \mathfrak{s}_{(\pmb a;\lambda)}(x).
    \]
    where $\mu=(a_m, \lambda_1,\dots,\lambda_\ell)$. 
\end{enumerate}
\end{proposition}
\begin{proof}
First suppose that $\lambda=\emptyset$. We have 
from proposition 23 in \cite{mSym} that the $m$-symmetric Schur function is such that
$$
s_{(\pmb a;\emptyset)}(x;t)=H_{\pmb a}(x;t).
$$
Recall that $H_{\pmb a}(x;t)$ is obtained from the dominant case $x^{\pmb a^+}$ by acting with a string of $T_i$ operators. Since $\lim_{t \to 0} T_i=\hat \pi_i$, this readily implies that
$$
 \mathfrak{s}_{(\pmb a;\emptyset)}(x_1,\dots,x_N)= s_{(\pmb a;\emptyset)}(x;0)=\lim_{t\to 0} H_{\pmb a}(x;t)=\hat K_{\pmb a}(x),
$$
given that $\hat K_{\pmb a}(x)$ is obtained from the dominant case $x^{\pmb a^+}$ by acting with a string of $\hat \pi_i$ operators instead of a string of $T_i$ operators.

We now prove the recursion 
\[
     \mathfrak{as}_{(a_1,\dots,a_{m-1} ; \mu)}(x)
    =
    W_{m-1}^{(N)}\,
    \mathfrak{as}_{(\pmb a;\lambda)}(x).
    \]
We have
$$
 W_{m-1}^{(N)} \mathfrak{s}_{  ( \pmb a; \lambda ) } ( x ) =
\pi_{N-1} \ldots \pi_{m} W_{m}^{(N)} \mathfrak{s}_{ ( \pmb a; \lambda )  } (x ) =
\pi_{N-1} \ldots \pi_{m} \mathfrak{s}_{ ( \pmb a; \lambda )  } ( x ), 
$$
since $\mathfrak{s}_{ ( \pmb a; \lambda )  } ( x )$ is symmetric in the variables $x_{m+1},\dots,x_N$. Therefore, using $\pi_i=1+\hat \pi_i$, we get
\begin{align*}
W_{m-1}^{(N)} \mathfrak{s}_{  ( \pmb a; \lambda ) } ( x ) & = 
(1+ \hat{\pi}_{N-1}) \ldots (1+ \hat{\pi}_{m}) \mathfrak{s}_{ ( \pmb a; \lambda ) } ( x ) \\
& =
\Big( 1+\hat \pi_m + \hat \pi_{m+1} \hat \pi_m+\cdots+\hat \pi_{N-1} \cdots \hat \pi_m
\Big)  \mathfrak{s}_{ ( \pmb a; \lambda )  } ( x ), 
\end{align*}
which again follows from the fact that $\mathfrak{s}_{ ( \pmb a; \lambda )  } ( x )$ is symmetric in the variables $x_{m+1},\dots,x_N$ given that $\hat \pi_i f=0$ if $f$ is symmetric in the variables $x_i$ and $x_{i+1}$. From Remark~\ref{InclusionRemark}, this yields
$$
W_{m-1}^{(N)} \mathfrak{s}_{  ( \pmb a; \lambda ) } ( x ) = \Big( 1+\hat \pi_m + \hat \pi_{m+1} \hat \pi_m+\cdots+\hat \pi_{N-1} \cdots \hat \pi_m
\Big)  \sum_\beta \hat K_{\pmb a,\beta}( x ),  
$$
where the sum is over all distinct permutations of $(\lambda_1,\dots,\lambda_\ell,0^{N-\ell})$. Hence, we finally have
$$
W_{m-1}^{(N)} \mathfrak{s}_{  ( \pmb a; \lambda ) } ( x ) = \Big( 1+\hat \pi_m + \hat \pi_{m+1} \hat \pi_m+\cdots+\hat \pi_{N-1} \cdots \hat \pi_m
\Big)  \sum_\beta \hat K_{\pmb a',a_m,\beta}( x ) =  \sum_\gamma \hat K_{\pmb a',\gamma}( x ),   
$$
where $\pmb a'=(a_1,\dots,a_{m-1})$, and where the sum is over all distinct permutations of 
$(a_m,\lambda_1,\dots,\lambda_\ell,0^{N-\ell-1})$. Note that there is no overcounting given that $\hat \pi_i \hat K_\omega(x)=0$ whenever $\omega_i=\omega_{i+1}$. Note also that we need the condition $a_m \geq \lambda_1$ for all 
distinct permutations of 
$(a_m,\lambda_1,\dots,\lambda_\ell,0^{N-\ell-1})$
to appear.  Using again Remark~\ref{InclusionRemark}, we immediately get that
$$
W_{m-1}^{(N)} \mathfrak{s}_{  ( \pmb a; \lambda ) } ( x ) = \mathfrak{s}_{  ( a_1,\dots,a_{m-1}; a_m,\lambda_1,\dots,\lambda_{\ell} ) } ( x ). 
$$
\end{proof}
We can now make explicit the connection between the almost symmetric Schur functions and the $m$-symmetric Schur functions at $t=0$.  For completeness, we admit here almost symmetric Schur functions indexed by pairs $(\pmb a,\lambda)$ whose last entry $a_m$ vanishes, so that they form a basis of $R_m$.
\begin{proposition} \label{propoMilo} We have 
\begin{equation}
\label{MiloEnLuc}
\mathfrak{as}_{(\pmb a ;\lambda)}(x_1,\dots,x_N)=
\sum_{(\pmb b;\mu) : K_+(\pmb b,\mu) \leq K_+(\pmb a,\lambda)} \mathfrak{s}_{(\pmb b;\mu)} (x_1,\dots,x_N),
\end{equation}
where the sum is over all $m$-partitions $(\pmb b;\mu)$.   
\end{proposition}
\begin{proof}
Concatenating $\pmb a$ and $\lambda$, 
and using the well known expansion of key polynomials into dual key polynomials,
we have from 
Definition~\ref{defirecursion} and Proposition~\ref{proporecursion} that
\begin{equation} \label{eqKcomp}
\mathfrak{as}_{(\pmb a,\lambda ; \emptyset)}(x)
=
K_{\pmb a,\lambda}(x)=
\sum_{\pmb b,\gamma : K_+(\pmb b,\gamma) \leq K_+(\pmb a,\lambda)} \hat K_{\pmb b,\gamma}(x),
\end{equation}
where the sum is over all pairs of compositions $\pmb b$ and $\gamma$ that satisfy the condition. 
Supposing the $\lambda$ is of length $\ell$, we have from Definition~\ref{defirecursion} that
$$
\mathfrak{as}_{(\pmb a ;\lambda)}(x) = W_{m}^{(N)} W_{m+1}^{(N)} \cdots W_{m+\ell-1}^{(N)} \mathfrak{as}_{(\pmb a,\lambda , \emptyset)}(x)= W_{m}^{(N)}  \mathfrak{as}_{(\pmb a,\lambda ; \emptyset)}(x).
$$
Hence, using \eqref{eqpiw0} and the fact that $\pi_i \hat K_{\alpha}(x)=0$ if $\alpha_i<\alpha_{i+1}$, we get from \eqref{eqKcomp} that
$$
\mathfrak{as}_{(\pmb a ;\lambda)}(x) = W_{m}^{(N)} 
\sum_{\pmb b,\gamma : K_+(\pmb b,\gamma) \leq K_+(\pmb a,\lambda)} \hat K_{\pmb b,\gamma}(x)=
\sum_{(\pmb b;\mu) : K_+(\pmb b,\mu) \leq K_+(\pmb a,\lambda)} \hat K_{\pmb b,\mu} (x),
$$
where the sum is now only over  pairs $\pmb b ,\mu$ such that $\mu$ is a partition. Since in this case $(\pmb b; \mu)$ is an $m$-partition,
 the proposition then follows immediately from Proposition~\ref{proporecursion}.
\end{proof}
As a corollary, we obtain a tableau generating function for the almost symmetric Schur functions.  It is not obviously the same as the combinatorial interpretation given in \cite{AlmostSym}, which rests on an HHL-type formula for key polynomials. 
\begin{corollary}
\label{CorMiloKey}
The almost symmetric Schur functions are such that
\begin{equation}
\label{SchurMilo2}
\mathfrak{as}_{(\pmb a ; \lambda)}( x ) := \sum_{ T \in \mathcal{AS} (\Lambda) } x^{T},
\end{equation}
where
\begin{equation}
\label{SetmSchurMilo}
\mathcal{AS} (\Lambda) =    \left \{ T  \in {\rm Tab}_{\{1,\dots,N \}}(\Lambda^{(0)}) \, : \, K_{+}  (T)  \leq K_{+} (\Lambda)  \right \}
\end{equation}
for $\Lambda=(\pmb a;\lambda)$.
\end{corollary}

\begin{example}
We now express the examples of almost symmetric Schur functions found in \cite{AlmostSym} in terms of the basis
$\{\mathfrak{s}_{\Lambda}(x)\}_{\Lambda}$ and illustrate the corresponding diagrams. In each case, the diagram on the left represents the almost symmetric Schur function, while those on the right correspond to the $m$-symmetric Schur functions appearing in its expansion.
\begin{enumerate}
    \item
$$
\begin{array}{l}
    \mathfrak{as}_{(01;2)} (x) = \\
    \\
    x_{1}^{2}x_{2} +x_{1}^{2} s_{1} (x_{3}) +x_{2}^{2}x_{1} +x_{2}^{2} s_{1} (x_{3})+x_{1} s_{2} (x_{3})  +x_{2} s_{2} (x_{3}) + 2x_{1}x_{2} s_{1}(x_{3})  = \\
    \\
     \mathfrak{s}_{(01;2)} ( x ) + \mathfrak{s}_{(10;2)} ( x ) + \mathfrak{s}_{(02;1)} ( x ) + \mathfrak{s}_{(20;1)} ( x ) + \mathfrak{s}_{(12;\emptyset)} ( x ) + \mathfrak{s}_{(21;\emptyset)} ( x )
\end{array}
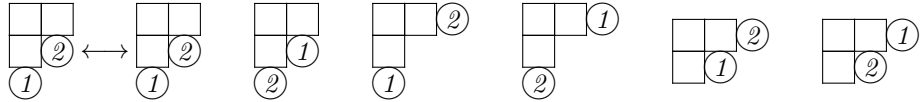
$$    
$$
\tableau[scY]{ &    \\ &   \bl \tcercle{2}     \\ \bl \tcercle{1} } \longleftrightarrow
\tableau[scY]{ &    \\ &   \bl \tcercle{2}     \\ \bl \tcercle{1} } \qquad
\tableau[scY]{ &    \\ &   \bl \tcercle{1}     \\ \bl \tcercle{2} } \qquad
\tableau[scY]{ &   &   \bl \tcercle{2}  \\     \\ \bl \tcercle{1} } \qquad
\tableau[scY]{ &   &   \bl \tcercle{1} \\      \\ \bl \tcercle{2} } \qquad
\tableau[scY]{ &   &   \bl \tcercle{2} \\  & \bl \tcercle{1}     } \qquad
\tableau[scY]{ &   &   \bl \tcercle{1} \\  & \bl \tcercle{2}     }  \qquad
$$
    \item $\mathfrak{as}_{(2;31)} ( x ) = x_{1}^{3} s_{21} (x_{2},x_{3}) +x_{1}^{2} s_{31} (x_{2},x_{3}) = \mathfrak{s}_{(2;31)}(x) + \mathfrak{s}_{(3;21)}(x)$
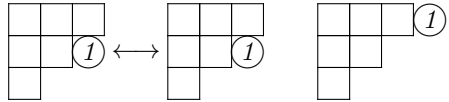
$$
\tableau[scY]{ &  &  \    \\ &   & \bl \tcercle{1}    \\ & \bl  } \longleftrightarrow 
\tableau[scY]{ &  &  \    \\ &   & \bl \tcercle{1}    \\ & \bl  } \qquad
\tableau[scY]{ &  &  & \bl \hspace{1.4mm} \tcercle{1} \    \\ &       \\ & \bl  }
$$    
    \item $\mathfrak{as}_{(21;1)} (x) = x_{1}^{2}x_{2} s_{1}(x_{3}) = \mathfrak{s}_{(21;1)} (x)$
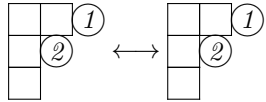
$$
\tableau[scY]{ &  & \bl \tcercle{1}    \\ &   \bl \tcercle{2}    \\ & \bl  } \longleftrightarrow 
\tableau[scY]{ &  & \bl \tcercle{1}    \\ &   \bl \tcercle{2}    \\ & \bl  }
$$
    \item $\mathfrak{as}_{(12;1)} (x) =  x_{1}^{2}x_{2} s_{1}(x_{3})+ x_{1}x_{2}^{2} s_{1}(x_{3})  = \mathfrak{s}_{(12;1)} (x) +  \mathfrak{s}_{(21;1)} (x) $
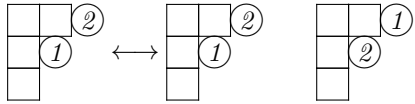
$$
\tableau[scY]{ &  & \bl \tcercle{2}    \\ &   \bl \tcercle{1}    \\ & \bl  } \longleftrightarrow \tableau[scY]{ &  & \bl \tcercle{2}    \\ &   \bl \tcercle{1}    \\ & \bl  } \qquad
\tableau[scY]{ &  & \bl \tcercle{1}    \\ &   \bl \tcercle{2}    \\ & \bl  } 
$$        

    \item $\mathfrak{as}_{(1;21)} (x)= x_{1}^{2} s_{11} (x_{2}+x_{3}) + x_{1} s_{21} (x_{2}+x_{3}) = \mathfrak{s}_{(1;21)} ( x ) + \mathfrak{s}_{(2;11)} ( x )$
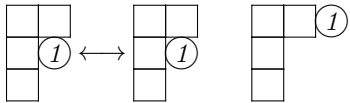
$$
\tableau[scY]{ &    \\ &   \bl \tcercle{1}    \\ & \bl  } \longleftrightarrow \tableau[scY]{ &    \\ &   \bl \tcercle{1}    \\ & \bl  }  \qquad \tableau[scY]{ &   & \bl \tcercle{1}   \\  & \bl  \\  & \bl   } 
$$    
\end{enumerate}    
\end{example}

\subsection{Duality and Cauchy identities}

The following generalization of the Cauchy identity
was obtained in Corollary~\ref{coroDual}
\begin{equation}
  \frac{1}{\left[\prod_{i+j \leq m+1}(1- x_i  y_j) \right] \left[\prod_{i,j=1}^N (1- x_i y_j) \right] }  
  =  \sum_{\Lambda}  \mathfrak s_{\Lambda}  (x_1,\dots,x_N)  \mathfrak{s}_{\Lambda}^{*}(y_1,\dots,y_N).
\end{equation}
Using Corollary~\ref{coroKeys}, it is then immediate that
\begin{equation*}
  \frac{1}{\left[\prod_{i+j \leq m+1}(1- x_i  y_j) \right] \left[\prod_{i,j=1}^N (1- x_i y_j) \right] }  
  =  \sum_{\Lambda}  \mathfrak h_{\Lambda}  (x_1,\dots,x_N)  \mathfrak{h}_{\Lambda}^{*}(y_1,\dots,y_N),
\end{equation*}
where 
$$\mathfrak h_\Lambda  (x_1,\dots,x_N) 
= \hat K_{\pmb a}(x_1,\dots,x_m) s_{\lambda}(x_1,\dots,x_N),
$$
and
$$\mathfrak h_\Lambda^*  (x_1,\dots,x_N) 
= K_{\omega_m(\pmb a)}(x_1,\dots,x_m) s_{\lambda}(x_1,\dots,x_N).
$$
We can thus define a natural scalar product on $R_m$ such that
\begin{equation} \label{eqSP}
    \langle \mathfrak h_\Lambda , \mathfrak h_\Omega^*  \rangle= \delta_{\Lambda \Omega} \qquad {\rm and} \qquad   \langle \mathfrak s_\Lambda , \mathfrak s_\Omega^*  \rangle= \delta_{\Lambda \Omega}.
\end{equation}
We are now in a position to introduce the functions that are dual  to the almost symmetric Schur functions with respect to this scalar product.

\begin{definition}
   The  almost symmetric dual Schur functions are defined as  
$$
\mathfrak{as}^{*}_{\Lambda} ( x ) =  \sum_{ T \in \mathcal {AS}^*(\Lambda) } x^{T} \big |_{\hat x_1 = x_1,\dots, \hat x_m = x_m }, 
$$
  where 
\begin{equation} \label{eqrestrictionMilo}
\mathcal {AS}^*(\Lambda) =  \left \{ T \in {\rm Tab}_{\mathcal A}(\Lambda^{(0)})\,  : \,  K_+(T) |_{\hat {\mathcal A}}  =  K_{+}^{*} (\Lambda) \big |_{\hat {\mathcal A}} \right \}. 
\end{equation}  
\end{definition}
We will need the following lemma, which says that the three natural ways of comparing the key tableaux attached to two $m$-partitions with the same underlying shape all agree.

\begin{lemma}
    Let $\Lambda$ and $\Omega$ be $m$-partitions such that $\Lambda^{(0)}=\Omega^{(0)}$. The following are then equivalent:
\begin{enumerate}
    \item $K_{+}^{*} (\Omega) \big|_{\hat{\mathcal{A}}}  \leq  K_{+}^{*} (\Lambda) \big|_{\hat{\mathcal{A}}}$
    \item $K_{+}^{*} (\Omega)  \leq  K_{+}^{*} (\Lambda)$
    \item $K_{+} (\Omega) \geq  K_{+} (\Lambda)$.
\end{enumerate}
\end{lemma}

\begin{proof}
The implication
\[
K_{+}^{*}(\Omega)\leq K_{+}^{*}(\Lambda)
\quad\Longrightarrow\quad
K_{+}^{*}(\Omega)\big|_{\hat{\mathcal A}}
\leq
K_{+}^{*}(\Lambda)\big|_{\hat{\mathcal A}}
\]
is immediate. Conversely, suppose that $
K_{+}^{*}(\Omega)\big|_{\hat{\mathcal A}}
\leq
K_{+}^{*}(\Lambda)\big|_{\hat{\mathcal A}}$. Then, in every column $r$ we have that  $
\hat{\mathcal C}^{*}_{r}(\Omega)
\leq
\hat{\mathcal C}^{*}_{r}(\Lambda)$, which implies that $\mathcal C^{*}_{r}(\Lambda)$ contains no more letters from
$\mathcal A\setminus\hat{\mathcal A}$ than
$\mathcal C^{*}_{r}(\Omega)$. Since these positions are filled with the largest possible entries of $\{1,\ldots,N\}$, we conclude that
\[
K_{+}^{*}(\Omega)\leq K_{+}^{*}(\Lambda).
\]

Finally,
\[
K_{+}^{*}(\Omega)\leq K_{+}^{*}(\Lambda)
\quad\Longleftrightarrow\quad
K_{+}(\Omega)\geq K_{+}(\Lambda),
\]
since the $^{*}$ operation is an order-reversing involution.
\end{proof}

The following proposition, which is analogous to Proposition~\ref{propoMilo},
is almost immediate.
\begin{proposition} \label{propoMiloDual} We have that 
    \begin{equation}
\mathfrak{s}_{(\pmb a ;\lambda)}^*(x)=
\sum_{(\pmb b;\mu) : K_+(\pmb b,\mu) \geq K_+(\pmb a,\lambda)} \mathfrak{as}_{(\pmb b;\mu)}^* (x_1,\dots,x_N),
\end{equation}
where the sum is over all $m$-partitions $(\pmb b;\mu)$.   
\end{proposition}
\begin{proof}
Recall from \eqref{eqrestriction} that 
$$
\mathfrak{s}^{*}_{\Lambda} ( x ) =  \sum_{ T \in \mathcal {S}^*(\Lambda) } x^{T} \big |_{\hat x_1 = x_1,\dots, \hat x_m = x_m }, 
$$
  where 
\begin{equation} 
\mathcal S^*(\Lambda) =  \left \{ T \in {\rm Tab}_{\mathcal A}(\Lambda^{(0)})\,  : \,  K_+(T)     \leq  K_{+}^{*} (\Lambda) \right \}. 
\end{equation}  
The proposition follows immediately from the preceding lemma.
\end{proof}

Finally, we prove that the almost symmetric dual Schur functions are indeed dual to the almost symmetric Schur functions with respect to the scalar product defined in \eqref{eqSP}.
\begin{proposition} We have that
$$
\langle \mathfrak {as}_\Lambda , \mathfrak {as}_\Omega^*  \rangle= \delta_{\Lambda \Omega}.
$$
\end{proposition}
\begin{proof} This is elementary linear algebra.  We have from Propositions~\ref{propoMilo} and \ref{propoMiloDual} that
$$
\mathfrak{s}^{*}_{\Lambda} = \sum_{\Omega } c_{\Lambda \Omega} \, \mathfrak{as}^{*}_{\Omega}   \qquad {\rm and} \qquad \mathfrak{as}_{\Lambda} = \sum_{\Omega } c_{\Omega \Lambda} \, \mathfrak{s}_{\Omega},  
$$
where the matrix $(c_{\Lambda \Omega})$ is invertible since it is triangular with 1's on the diagonal. Denoting by $(\bar c_{\Lambda \Omega})$ the inverse matrix, we then get
$$
\langle \mathfrak {as}_\Lambda , \mathfrak {as}_\Omega^*  \rangle= \langle \sum_{\Gamma } c_{\Gamma \Lambda} \, \mathfrak{s}_{\Gamma} ,  \sum_{\Delta } \bar c_{ \Omega \Delta} \, \mathfrak{s}_{\Delta}^*  \rangle =\sum_{\Gamma,\Delta } c_{ \Gamma \Lambda} \, \bar c_{ \Omega \Delta} \, \delta_{\Gamma \Delta}=\sum_{\Gamma} c_{ \Gamma \Lambda} \, \bar c_{ \Omega \Gamma} = \delta_{\Lambda \Omega}.
$$
\end{proof}

\subsection{Determinantal formulas for the $m$-Schur functions and the almost symmetric Schur functions}

By Definition~\ref{defdualS}, the function $ \mathfrak{s}^{*}_{\Lambda}(x)$ is a multi-Schur function when $\Lambda$ is dominant, and thus already comes with a determinantal formula.  In this subsection, we derive similar  determinantal formulas for 
$\mathfrak{s}_{\Lambda}(x)$, $\mathfrak{as}_{\Lambda}(x)$, and $\mathfrak{as}^{*}_{\Lambda}(x)$
when $\Lambda$ is dominant (or antidominant in the case of $\mathfrak{as}^{*}_{\Lambda}(x)$). 

Since this section features several Jacobi-Trudi type determinants, we will adopt a compact notation. We display only the diagonal entries, with the understanding that the alphabet remains constant along each row, and that the degree increases by one from left to right.

\begin{example}
    With this convention, Example~\ref{ExamplemSchurdual} can be written as
   $$
s_{2,1;3,1}^*(x;t)= \det \left(
\begin{array}{cccc}
  h_3[X] & h_2[X+x_1] & h_1 [X+x_1+x_2]& h_1[X+x_1+x_2]
\end{array}   
\right).
$$
\end{example}

All three formulas are obtained by the same route.  We start from a flagged Schur expansion of the determinant, due to Wachs \cite{Wachs1985}, which produces a set of tableaux described by conditions on the columns in which each letter may appear.  Lemma~\ref{LemmaLeftKeys} turns such conditions into conditions on the \emph{left} key, and the Schützenberger involution then converts the latter into the conditions on the \emph{right} key that define the sets $\mathcal S(\Lambda)$ and $\mathcal{AS}(\Lambda)$.  We begin by recalling that involution.

\begin{definition}
Let $T$ be a tableau on the alphabet $\{1,\ldots,N\}$ with column reading word
\[
w(T)=w_1\cdots w_r.
\]
The \emph{Schützenberger involution}, or \emph{evacuation}, of $T$ is the unique tableau of straight shape, denoted by $\omega_S(T)$, whose column reading word is
\[
(N+1-w_r)\cdots(N+1-w_1).
\]
\end{definition}

The property of the Schützenberger involution that we will need is the following.  Although it is implicit in Lascoux's work, a proof in type $C$ appears in Proposition~64 of \cite{MiguelSantos2021}, and the same argument applies mutatis mutandis in type $A$.
\begin{proposition}
\label{Evacuation}
    Let $T$ be a tableau, and let $\omega_{S}$ denote the Schützenberger involution. Then
$$
\omega_{S} (K_{-}(T)) = K_{+} (  \omega_{S}(T)).
$$
\end{proposition}

\begin{example}
Let $N=4$ and
\[
T=\tableau[scY]{ 1 & 1 & 2 \\ 3 }.
\]
Then
\[
K_{-}(T)=\tableau[scY]{ 1 & 1 & 1 \\ 3 },
\]
and the reading word of $T$ is
\[
w(T)=3112.
\]

Applying the Schützenberger involution, we obtain
\[
\omega_S(w(T))=2443
\qquad\text{and}\qquad
\omega_S(T)=\tableau[scY]{ 2 & 3 & 4 \\ 4 }.
\]
Consequently,
\[
K_{+}\bigl(\omega_S(T)\bigr)
=
\tableau[scY]{ 2 & 4 & 4 \\ 4 }
=
\omega_S\bigl(K_{-}(T)\bigr).
\]
\end{example}

\begin{lemma}
\label{LemmaEvac}
Let $Q$ be a tableau in
the ordered alphabet
\[
N<N-1<\cdots<m<\cdots<2<1
\]
such that, for $i=1,\dots,m$, letter $i$ occurs $a_i$ times in $Q$.
If, for $i=1,\dots,m$, letter $i$ appears in $Q$ in the first $a_i$ columns and nowhere else, then 
 $\omega_S(Q)$ contains, for $i=1,\dots,m$, the entry $i$ in each of the first $a_i$ columns and nowhere else, where the tableau $\omega_S(Q)$ is in the ordered alphabet
\[
1<2<\cdots<m<\cdots<N-1<N.
\]
\end{lemma}
\begin{proof}
While the letters $N,\dots,m+1$ are being evacuated, the letters $m$ to $1$ can only move weakly to the left.  Once those letters have been pushed out, the letters $m$ to $1$ therefore form the unique tableau in the ordered alphabet $m<\cdots<1$, of shape $\pmb a^+=(a_1,\dots,a_m)^+$, in which each letter $i$ occurs $a_i$ times.  When the evacuation is complete, the letters $1$ to $m$ likewise form the unique tableau in the ordered alphabet $1<\cdots<m$, of shape $\pmb a^+$, in which each letter $i$ occurs $a_i$ times.  This is exactly the required configuration.

\end{proof}

\begin{lemma}
\label{LemmaLeftKeys}
Let $T$ be a tableau in
the ordered alphabet
\[
N<N-1<\cdots<m<\cdots<2<1
\]
such that, for $i=1,\dots,m$, letter $i$ occurs $a_i$ times in $T$. If $\pmb a =(a_1,\dots,a_m)$ is dominant,  the tableau $T$ is then  such that $K_-(T)$ contains the entry $i$ in each of the first $a_i$ columns and nowhere else if and only if the tableau $T$ itself has an $i$ in each of the first $a_i$ columns while it has no $i$ in column $a_i+1$.
\end{lemma}

\begin{proof}
The Lascoux action used to construct $K_-(T)$ from $T$ is such that if the letter $i$ occurs in each of the first $\ell$ columns of $T$, then $K_-(T)$ also has a letter $i$ in each of its first $\ell$ columns.  The necessity of the condition follows at once: $K_-(T)$ can only have the entry $i$ in its first $a_i$ columns and nowhere else if $T$ itself has an $i$ in each of the first $a_i$ columns while having none in column $a_i+1$.

It thus remains to prove that such a $T$ always produces a $K_-(T)$ of the desired form.  By the observation above, $K_-(T)$ does contain an $i$ in each of its first $a_i$ columns, so that only the absence of further $i$'s needs to be established.

Because $\pmb a$ is dominant, for each $j\le a_i$, the $j$-th column of $T$ contains the entries
\[
i,i-1,\dots,2,1.
\]

Before continuing, let us observe that, as a consequence of Proposition~\ref{IgualdadKeys}, when we permute  columns $C_{1}$ and $C_{2}$ (with $C_1$ to the left of $C_2$) to get $C_{1}'$ and $C_{2}'$, column $C_{2}'$ is equal to the supremum of the column reading word of $C_{1}C_{2}$.

We will now see that  column $a_{i}+1$ of $K_{-}(T)$ does not have $i$ as an entry. Let $a_{j}, a_{j+1}, \ldots , a_{k}$ be all the entries of $\pmb a$ that are equal to $a_{i}$ (with $j \leq i \leq  k$).
Let $C_{1}$ (resp. $C_2$) be column $a_{i}$ (resp. $a_i+1$) of $T$.  Column $C_{1}$ contains the entries $k, \ldots, 1$, while  $C_{2}$ only contains $j-1,\ldots,1$.  As the supremum of those two columns starts with $1,2,\ldots,k$, column $C_{2}'$ contains all those entries, which implies that $C_{1}'$ does not have  the entries $j, \ldots, k$. Repeating this process until $C_{1}'$ is the first column of $T$, we get that column $a_{i}+1$ of $K_{-}(T)$ does not have the entries $j, \ldots, k$, which implies in particular that it does not contain the letter $i$.

The left key $K_-(T)$ therefore has the entry $i$ in the first $a_i$ columns and nowhere else, which is what we had to prove.
\end{proof}

\begin{proposition}
\label{NoDualSchurDom}
For every dominant $m$-partition $\Lambda=(\pmb a;\lambda)$, the $m$-symmetric Schur function at $t=0$ admits the determinantal expression
\begin{equation}
{s}_{\Lambda}(x;0)=\mathfrak{s}_{\Lambda}(x)
=
\left(\prod_{i=1}^{m} x_i^{a_i}\right)
\det\!\left(
e_{\lambda_i'-i+j}
\left[
X+X_{i}
\right]
\right)_{1 \leq i,j \leq \ell (\Lambda)},
\end{equation}
where 
$$
X_{i} = \sum_{j \, : \, a_j < i } x_j .
$$
\end{proposition}

\begin{proof}
Fix the ordered alphabet
\[
N<N-1<\cdots<2<1.
\]
By Theorem~3.5* of \cite{Wachs1985}, the right-hand side of the identity in Proposition~\ref{NoDualSchurDom} is equal to 
\[
\left(\prod_{i=1}^{m} x_i^{a_i}\right)\sum_T x^T,
\]
where the sum ranges over all tableaux $T$ of shape $\lambda$ such that, for each $i=1,\dots,m$, letter $i$ does not appear in any of the first $a_{i}+1$ columns.

Given such a tableau $T$, append an entry $i$ to each of the first $a_i$ columns. This defines a weight-preserving bijection
\[
\left(\prod_{i=1}^{m} x_i^{a_i}\right)\sum_T x^T
=
\sum_{T'} x^{T'},
\]
where the sum on the right ranges over all tableaux $T'$ of shape $\pmb a \cup \lambda$ such that for each $i=1,\dots,m$, letter $i$ appears in each of the first $a_i$ columns while it does not appear in column $a_i+1$. 

By Lemma~\ref{LemmaLeftKeys}, these are precisely the tableaux whose left key has letter $i$  in each of the first $a_i$ columns and nowhere else. We then get from Lemma~\ref{LemmaEvac} and Proposition~\ref{Evacuation} that the set $\{ \omega_S(T')\}$  is  the set of tableaux on the alphabet
\[
1<2<\cdots<N
\]
such that, for each $i$, the right key $K_+(\omega_S (T'))=\omega_S(K_-(T'))$ contains the entry $i$  in each of the first $a_i$ columns and nowhere else. But this is the same as saying that $\omega_S (T') \in \mathcal S(\Lambda)$, where
$\mathcal S(\Lambda)$ was introduced in \eqref{SetmSchur}.
Therefore, we get from \eqref{mschur2T} that
\[
\left(\prod_{i=1}^{m} x_i^{a_i}\right)\sum_T x^T
=
\sum_{T'} x^{T'}
=
\sum_{\omega_S (T') \in \mathcal S(\Lambda)} x^{\omega_S (T')}=\mathfrak s_{\Lambda}(x).
\]
\end{proof}

\begin{example}
\label{ExmSchurDom}
   Let $\Lambda = (2,1;4)$. The Young diagram associated to $\Lambda$ is then
$$
    \tableau[scY]{ & & &   \\   &  & \bl \tcercle{1} \\   &
\bl \tcercle{2}  
}
.
$$
and we have
$$
\mathfrak{s}_{\Lambda} (x) 
=
x_{1}^{2}x_{2}
\det
\begin{pmatrix}
e_{1}[X] & e_{1} [X] &  e_{1}[X+x_{2}] & e_{1}[X+x_{1}+x_{2}]  
\end{pmatrix}.
$$
\end{example}

\begin{proposition}
For every dominant $(\pmb a,\lambda)$, the almost symmetric Schur function admits the determinantal expression
\begin{equation}
\label{MiloDom}
\mathfrak{as}_{(\pmb a;\lambda)}(x)
=
\left(\prod_{i=1}^{m} x_i^{a_i}\right)
\det\!\left(
e_{\lambda_i'-i+j}
\left[
X+X_{i}
\right]
\right)_{1 \leq i,j \leq \ell (\Lambda)},
\end{equation}
where
$$
X_{i} = \sum_{j: a_{j} \leq  i } x_{j} .
$$
\end{proposition}

\begin{proof}
The argument closely parallels that of Proposition~\ref{NoDualSchurDom}, the only change being in the flag.  In the flagged Schur expansion produced by Theorem~3.5* of \cite{Wachs1985}, the entry $i$ may now already appear in column $a_{i}+1$ of $T$, whereas in Proposition~\ref{NoDualSchurDom} it could only appear \emph{strictly after} that column.

After multiplying by $\prod_{i=1}^m x_i^{a_i}$, we therefore obtain a generating function over the set of tableaux $T'$ in which, for each $i$, the entry $i$ appears in all of the first $a_i$ columns and is \emph{free to appear or not} in the remaining ones.  By Lemma~\ref{LemmaLeftKeys}, this says exactly that the left key of $T'$ contains $i$ at least in the first $a_i$ columns.

Applying the Schützenberger involution to each $T'$ and using Proposition~\ref{Evacuation}, we obtain a weight-preserving bijection from this set onto the set of tableaux $T'' \in {\rm Tab}_{\{1,\dots,N \}}(\Lambda^{(0)})$ satisfying
\[
K_+(T'')
\;\le\;
K_+(\Lambda),
\]
which is to say that $T'' \in \mathcal{AS}(\Lambda)$.  These tableaux are therefore precisely those contributing to $\mathfrak{as}_{(\pmb a ; \lambda)}(x)$, which proves the proposition.
\end{proof}

\begin{corollary}
    For a dominant $m$-partition $\Lambda$, the functions $\mathfrak{s}_{\Lambda}(x)$ and $\mathfrak{as}_{\Lambda}(x)$ are the product of a column flagged Schur function and a monomial $x^{\pmb a}$.
\end{corollary}

\begin{example}

   Let $\Lambda = (2,1;4)$ as in Example~\ref{ExmSchurDom}. We then have
$$
\mathfrak{as}_{\Lambda} (x) 
=
x_{1}^{2}x_{2}
\det
\begin{pmatrix}
e_{1}[X] & e_{1} [X+x_{2}] &  e_{1}[X+x_{1}+x_{2}] & e_{1}[X+x_{1}+x_{2}]  
\end{pmatrix}.
$$
\end{example}

\begin{proposition}
Let $\Lambda=(\pmb a;\lambda)$ be an antidominant $m$-partition. Assume that $\pmb a$ is nonzero, and let $a_{r}$ be the first nonzero entry of $\pmb a$. Set $\pmb b= (a_{m},\ldots, a_{r})$, which has length $k=m-r+1$ and is dominant because $\Lambda$ is antidominant, and let
\[
\pmb b^{-} =
(b_{1}-1,\ldots,b_{k}-1).
\]
Define
\[
\mu=\pmb b^{-}\cup\lambda.
\]
Then
\begin{equation}
\label{SchurMiloDet}
\mathfrak{as}^{*}_{\Lambda}(x)
=
\left(\prod_{i=1}^{k} x_i\right)
\det\!\bigl(
h_{\mu_i-i+j}[X+X_i]
\bigr)_{1 \leq i,j \leq \ell (\Lambda)},
\end{equation}
where
\[
X_i=  \sum_{j:b_{j} \ge \Lambda^{(0)}_i } x_{j}  .
\]
\end{proposition}

\begin{proof}
By Theorem~3.5 of \cite{Wachs1985}, the determinant in \eqref{SchurMiloDet} is the generating function
\[
\sum_T x^T,
\]
where $T$ ranges over all semistandard tableaux of shape $\mu$ with flag $\pmb c=(c_1,\dots,c_k)$, where $c_i$ is the index of the row of the diagram of $\Lambda$ containing the circle filled with an $i$. We emphasize that the flag restrictions are imposed on the alphabet $\hat{\mathcal{A}}$.

Given such a tableau $T$, construct a tableau $T'$ by appending an entry $\hat{i}$ in column $b_i$ for each $1\le i\le k$. This operation preserves the semistandard conditions: the flag restrictions imply that every entry in the row of the new cell in column $b_i$ is at most $\hat{i}$, while every entry above that row is strictly smaller than $\hat{i}$. Consequently, $T'$ is a semistandard tableau of shape $\pmb b \cup\lambda$. The factor $\prod_{i=1}^{k}x_i$ 
in \eqref{SchurMiloDet}
accounts  for the weights of the newly added entries. 

Let $\mathcal B$ denote the set of tableaux obtained from the construction described in the previous paragraph. 
 Let $T \in \mathcal{B}$; we show that $T \in \mathcal{AS}^{*} (\Lambda)$. We will see that, for every $t$, $\sup \hat{W} (T_{t}) = \hat{i} \cdots \hat{2} \hat{1}$, where $i$ is the largest integer such that the entry $\hat{i}$ is allowed to lie in $T_{t}$ (because of the flag conditions).

By construction of $\mathcal{B}$, for every $j \leq i$ the entry $\hat{j}$ appears in column $b_{j}$ (which is weakly to the right of column $t$). Consequently, since $\pmb b$ is dominant, the decreasing word $w = \hat{i}\,\cdots \hat{2}\,\hat{1}$ belongs to $\hat{W}(T_{t})$. Given that $w$ is the largest word in $W(\hat{\mathcal{A}})$ without entries larger than $\hat{i}$, it follows that $w = \sup  \hat{W}(T_{t})$.
Hence, $\mathcal{B} \subseteq \mathcal{AS}^{*} (\Lambda)$.

We now prove the other inclusion. Let $T \in \mathcal{AS}^{*}(\Lambda)$, fix $i$ and let $t=b_i$. Let $q$ be the largest integer such that $b_{i}=b_{q}$. Then
\begin{equation}
\label{SupLastProof}
    \hat{q} \ldots \hat{i} \ldots \hat{2} \hat{1} = \sup \hat{W} (T_{t}).
\end{equation}
In particular, $\hat{i}$ belongs to $T_{t}$. We claim that the entry $\hat{i}$ must lie exactly in column $t$. If $t$ is the last column of $T$, then it is immediate. Assume now that $T$ has a column $t+1$. Due to \eqref{SupLastProof}, we know that $\hat{i}$ belongs to $T_{t}$. On the other hand, since $b_{i}<t+1$,
$$
 \sup \hat{W} (T_{t+1})=\hat{j} \cdots \hat{2} \hat{1}.
$$
for some $j<i$, which implies that $\hat{i}\notin T_{t+1}$. Therefore, $\hat{i}$ must appear in column $t$, as claimed. It follows that every tableau $T \in \mathcal{AS}^{*}(\Lambda)$ contains exactly the distinguished entries prescribed in the construction of $\mathcal{B}$. Hence,
 $\mathcal{AS}^{*}(\Lambda)\subseteq \mathcal B$. 

Combining the two inclusions gives
$
\mathcal B=\mathcal{AS}^{*}(\Lambda)
$.
Since $\mathfrak{as}^{*}_{\Lambda}(x) = \sum_{T \in \mathcal{AS}^{*}(\Lambda)} x^{T}$, this concludes the proof.
\end{proof}

\begin{remark}
If $\pmb a=\pmb0$, then the product is empty and $\mu=\lambda$. Hence
\[
\mathfrak{as}^{*}_{\Lambda}(x)
=
\det\!\bigl(
h_{\lambda_i-i+j}[X]
\bigr)_{1 \leq i,j \leq \ell (\Lambda)}
=
s_\lambda[X],
\]
by the classical Jacobi-Trudi identity. Thus the proposition also holds in this case.
\end{remark}

\begin{example}
\label{EjemploDetMilo}
If $\Lambda = ( 2,3 ; 5 )$, we get
$$
\mathfrak{as}^{*}_{ ( 2,3 ; 5  )} (x) =
x_{1}x_{2}
\det
\begin{pmatrix}
 h_{5}[X] & h_{2}[X+x_{1}] &  h_{1}[X+x_{1}+ x_{2}] 
\end{pmatrix}.
$$    
Observe that this can be rewritten as
$$
\mathfrak{as}^{*}_{ ( 2,3 ; 5  )} (x) =
\det
\begin{pmatrix}
 h_{5}[X] & x_{1}h_{2}[X+x_{1}] &  x_{2}h_{1}[X+x_{1}+ x_{2}] 
\end{pmatrix},
$$
which has a similar a form as  
$$
\mathfrak{s}^{*}_{ ( 3,2 ; 5  )} (x) =
\det
\begin{pmatrix}
 h_{5}[X] & h_{3}[X+x_{1}] &  h_{2}[X+x_{1}+ x_{2}] 
\end{pmatrix}.
$$
\end{example}

\bibliographystyle{plain}
\bibliography{Biblio}

@article{AlmostSym,
  author  = {Bechtloff Weising, Milo},
  title   = {Almost symmetric schur functions},
  note    = {arXiv:2405.01049}
}

@article{ConchaLapointe,
  author  = {Concha, Manuel and Lapointe, Luc},
  title   = {The $m$-symmetric {M}acdonald polynomials},
  note    = {arXiv:2311.12625}
}

@article{mSym,
  author  = {Lapointe, Luc},
  title   = {$m$-symmetric functions, non-symmetric {M}acdonald polynomials and positivity conjectures},
  journal = {Trans. Amer. Math. Soc.},
  volume  = {378},
  number  = {12},
  pages   = {8319--8359},
  year    = {2025}
}

@article{Demazure1974a,
  author  = {Demazure, Michel},
  title   = {D\'esingularisation des vari\'et\'es de {S}chubert g\'en\'eralis\'ees},
  journal = {Ann. Sci. \'Ecole Norm. Sup. (4)},
  volume  = {7},
  pages   = {53--88},
  year    = {1974}
}

@article{Demazure1974b,
  author  = {Demazure, Michel},
  title   = {Une nouvelle formule des caract\`eres},
  journal = {Bull. Sci. Math. (2)},
  volume  = {98},
  number  = {3},
  pages   = {163--172},
  year    = {1974}
}

@incollection{Lascoux1990a,
  author    = {Lascoux, Alain and Sch\"utzenberger, Marcel-Paul},
  title     = {Keys \& standard bases},
  booktitle = {Invariant Theory and Tableaux (Minneapolis, MN, 1988)},
  series    = {IMA Vol. Math. Appl.},
  volume    = {19},
  pages     = {125--144},
  publisher = {Springer, New York},
  year      = {1990}
}

@article{FlaggedAreKey1995,
  author  = {Reiner, Victor and Shimozono, Mark},
  title   = {Key polynomials and a flagged {L}ittlewood-{R}ichardson rule},
  journal = {J. Combin. Theory Ser. A},
  volume  = {70},
  number  = {1},
  pages   = {107--143},
  year    = {1995}
}

@article{Wachs1985,
  author  = {Wachs, Michelle L.},
  title   = {Flagged {S}chur functions, {S}chubert polynomials, and symmetrizing operators},
  journal = {J. Combin. Theory Ser. A},
  volume  = {40},
  number  = {2},
  pages   = {276--289},
  year    = {1985}
}

@article{Willis2013,
  author  = {Willis, Matthew J.},
  title   = {A direct way to find the right key of a semistandard {Y}oung tableau},
  journal = {Ann. Comb.},
  volume  = {17},
  number  = {2},
  pages   = {393--400},
  year    = {2013}
}

@article{MiguelSantos2021,
  author  = {Santos, Jo\~ao Miguel},
  title   = {Symplectic keys and {D}emazure atoms in type {C}},
  journal = {Electron. J. Combin.},
  volume  = {28},
  number  = {2},
  pages   = {Paper No. 2.29, 42 pp.},
  year    = {2021}
}

@book{Fulton1996,
  author    = {Fulton, William},
  title     = {Young Tableaux. {W}ith Applications to Representation Theory and Geometry},
  series    = {London Mathematical Society Student Texts},
  volume    = {35},
  publisher = {Cambridge University Press, Cambridge},
  year      = {1997}
}

@book{Stanley_Fomin_1999,
  author    = {Stanley, Richard P.},
  title     = {Enumerative Combinatorics. {V}ol. 2},
  series    = {Cambridge Studies in Advanced Mathematics},
  volume    = {62},
  note      = {With a foreword by Gian-Carlo Rota and appendix 1 by Sergey Fomin},
  publisher = {Cambridge University Press, Cambridge},
  year      = {1999}
}

@misc{Positivityt1,
  author  = {Lapointe, Luc and Pena, Luis},
  title   = {A proof of the $m$-symmetric {M}acdonald positivity at $t=1$},
  note    = {arXiv:2608.10223}
}

@incollection{Lascoux2003,
  author    = {Lascoux, Alain},
  title     = {Double {C}rystal {G}raphs},
  booktitle = {Studies in Memory of {I}ssai {S}chur},
  series    = {Progr. Math.},
  volume    = {210},
  pages     = {95--114},
  publisher = {Birkh\"auser Boston, Boston, MA},
  year      = {2003}
}

@article{FuLascoux2009,
  author  = {Fu, Amy M. and Lascoux, Alain},
  title   = {Non-symmetric {C}auchy kernels for the classical groups},
  journal = {J. Combin. Theory Ser. A},
  volume  = {116},
  number  = {4},
  pages   = {903--917},
  year    = {2009}
}

@article{AzenhasEmami2015,
  author  = {Azenhas, Olga and Emami, Aram},
  title   = {An analogue of the {R}obinson-{S}chensted-{K}nuth correspondence and non-symmetric {C}auchy kernels for truncated staircases},
  journal = {European J. Combin.},
  volume  = {46},
  pages   = {16--44},
  year    = {2015}
}

@article{ChoiKwon2018,
  author  = {Choi, Seung-Il and Kwon, Jae-Hoon},
  title   = {{L}akshmibai-{S}eshadri paths and non-symmetric {C}auchy identity},
  journal = {Algebr. Represent. Theory},
  volume  = {21},
  number  = {6},
  pages   = {1381--1394},
  year    = {2018}
}

@article{Mason2009,
  author  = {Mason, Sarah},
  title   = {An explicit construction of type {A} {D}emazure atoms},
  journal = {J. Algebraic Combin.},
  volume  = {29},
  number  = {3},
  pages   = {295--313},
  year    = {2009}
}

@book{Macdonald1995,
  author    = {Macdonald, Ian G.},
  title     = {Symmetric Functions and {H}all Polynomials},
  edition   = {Second},
  series    = {Oxford Mathematical Monographs},
  publisher = {The Clarendon Press, Oxford University Press, New York},
  year      = {1995}
}

@article{Haiman2001,
  author  = {Haiman, Mark},
  title   = {{H}ilbert schemes, polygraphs and the {M}acdonald positivity conjecture},
  journal = {J. Amer. Math. Soc.},
  volume  = {14},
  number  = {4},
  pages   = {941--1006},
  year    = {2001}
}

\end{document}